%% file: March2026-version.tex
\documentclass{amsart}

\input{macro-2}

\title[Rational Functions on the Projective Line]{Rational Functions on the Projective Line from a Computational Viewpoint}

\author{Eslam Badr}
\address{Department of Mathematics,   American University, Cairo, Egypt}
\email{eslammath@aucegypt.edu}

\author{Elira Shaska}
\address{Department of Computer Science,  Oakland University, Rochester, MI, 48309}
\email{elirashaska@oakland.edu}

\author{Tony Shaska}
\address{Department of Mathematics,   Oakland University, Rochester, MI, 48309}
\email{tanush@umich.edu}

\begin{document}

\begin{abstract}
An explicit invariant-theoretic description of the moduli space \(\cM_3^1\) of degree-three rational maps on \(\P^1\) is developed. A cubic map \(\phi\) is represented, up to conjugation, by the pair of binary forms \((f, g) \in V_4 \oplus V_2\) arising from its Clebsch–Gordan decomposition. From this representation one constructs weighted projective invariants \(\xi_0,\dots,\xi_5\) that embed \(\cM_3^1\) into \(\P_{(2,2,3,3,4,6)}^5\) onto the locus where the gcd of the weights of the non-zero coordinates equals \(1\), together with absolute invariants defined as weight-zero rational functions of the \(\xi_i\), normalized by an additional invariant \(I_6\) of weight \(6\). These absolute invariants determine the isomorphism class uniquely.

The stratification of \(\cM_3^1\) is described explicitly by equations in the absolute invariants or polynomial relations among the \(\xi_i\). Computational illustrations demonstrate that the resulting invariants provide an effective feature set for automated classification of automorphism groups. The methods suggest natural extensions to higher degrees.
\end{abstract}



\maketitle

\input{body}

%*********************************************
%\bibliographystyle{amsplain}
%\bibliographystyle{amsalpha}
\bibliography{ref-2}

\include{appendix}
\end{document}

%% file: macro-2.tex
\usepackage[T1]{fontenc}

\usepackage[usenames,dvipsnames]{color}
\usepackage{tikz}
\usetikzlibrary{positioning,shadows,arrows}

\usepackage{amsmath}
\usepackage{amssymb,amsfonts,amsthm,amsbsy,amsopn,amstext,amsxtra,euscript,amscd,latexsym}
\usepackage{mathtools}

\usepackage{xy}
\input xy \xyoption{all}

\usepackage[shortlabels]{enumitem}

\usepackage{algorithm}
\usepackage{algorithmic}

\usepackage{booktabs}

\usepackage[initials, msc-links]{amsrefs}   %citation-order, backrefs, 

\usepackage{hyperref}
\hypersetup{
  colorlinks = true,
  urlcolor   = blue,
  linkcolor  = blue,
  citecolor  = blue
} 

\usepackage[capitalise]{cleveref}

\theoremstyle{plain}
\newtheorem {thm }{Theorem}
 \newtheorem*{thm*}{Theorem} % Define an unnumbered theorem environment

\newtheorem{prop}{Proposition}
\newtheorem{lem}{Lemma}

\newtheorem{rem}{Remark}
\newtheorem{exa}{Example}

\theoremstyle{definition}
\newtheorem{defi}{Definition}

\theoremstyle{remark}

\crefname{thm}{Thm.}{}
\crefname{prop}{Prop.}{}
\crefname{lem}{Lem.}{}
\crefname{cor}{Cor.}{}

\newcommand{\Z}{\mathbb Z}

\newcommand\Q{\mathbb Q}
\newcommand\C{\mathbb C}
\renewcommand\P{\mathbb P}

\renewcommand\a{\alpha}
\renewcommand\b{\beta}

 \newcommand\w{\omega}
 \newcommand\s{\sigma}

\newcommand\M{\mathcal M}

\newcommand\cS{\mathcal S}
\newcommand\cT{\mathcal T}

\newcommand\<{\langle}
\def\>{\rangle}

\DeclareMathOperator\PGL{PGL}

\DeclareMathOperator\Aut{Aut}
\DeclareMathOperator\Res{Res}
\DeclareMathOperator\Rat{Rat}

\DeclareMathOperator{\GL}{GL}
\DeclareMathOperator{\SL}{SL}

\DeclareMathOperator{\cR}{\mathcal R}

\DeclareMathOperator{\I}{\mathcal I}
\DeclareMathOperator{\J}{\mathcal J}

\newcommand\cN{\mathcal N}
\newcommand\RR{\mathcal R}
\newcommand\A{\mathbb A}

\newcommand\p{\mathfrak p}

\newcommand\iso{\cong}
\renewcommand\l{\lambda}

\DeclareMathOperator\ord{ord }

\DeclareMathOperator\Proj{Proj }

\DeclareMathOperator\Fix{\mathbf{Fix} }
\DeclareMathOperator\F{\mathbf F}

\DeclareMathOperator\G{\mathbf G}

\DeclareMathOperator\chara{char }
\DeclareMathOperator\lcm{lcm}

\def\L{\mathcal L}

\def\s{\sigma}
\def\c{c}

\newcommand\Img{\mathrm{Img}}

\newcommand\cP{\mathcal P}
\newcommand\cM{\mathcal M}

\usepackage{listings}

%% file: body.tex
%************************
 
%************************
\section{Introduction}\label{sec:intro}
Let $k$ be an algebraically closed field of characteristic zero and $\P_k^1$ the projective line over $k$.
A degree $d\geq 2$ rational function $\phi: \P^1 \to \P^1$ is given as the ratio of two degree $d$ binary forms, say $\phi(x, y) = \frac{f_0(x, y)}{f_1(x, y)}$ such that the resultant between $f_0(x, y)$ and $f_1(x, y)$ is non-zero. Hence, a rational function is a pair of binary forms of the same degree with no common roots. If we denote
$f_0(x, y) = \sum_{i=0}^d a_i x^i y^{d-i}$
and
$f_1(x, y) = \sum_{i=0}^d b_i x^i y^{d-i},$
then the collection of pairs $[f_0 : f_1]$ can be parametrized via
$[a_d : \cdots : a_0 : b_d : \cdots : b_0] \in \P^{2d+1},$
such that $\Res(f_0, f_1) \neq 0$. So the parameter space of degree $d>1$ rational functions on $\P^1$ is the complement of the resultant locus in $\P^{2d+1}$, say $\Rat_d^1 := \P^{2d+1}\setminus V(\Res)$.

The group $\SL_2(k)$ acts on $\Rat_d^1$ by conjugation, i.e., for some $M \in \SL_2(k)$,
$
\phi \to \phi^M := M^{-1} \circ \phi \circ M.
$
Two rational functions $\phi, \psi \in \Rat_d^1$ are called \textbf{conjugate} if there is an $M\in \SL_2(k)$ such that $\phi=\psi^M$.
The moduli space of degree $d>1$ rational functions (in one variable) is denoted by $\M_d^1$ and can be constructed as a quotient space of this $\SL_2$-action.

The automorphism group of $\phi$ is defined as
\[
\Aut(\phi) := \{ \sigma \in \PGL_2(k) \mid \phi^{\sigma}=\phi \}.
\]
It is a finite subgroup of $\PGL_2(k)$, so it is isomorphic to one of the following: a cyclic group $C_n$, a dihedral group $D_n$, $A_4$, $S_4$, or $A_5$. Determining which one of these groups occur for a fixed degree $d\geq 2$ is studied in \cite{miasnikov, deg-3-4, faria}.

For any $\phi(x, y) = \frac{f_0(x, y)}{f_1(x, y)}$, we define degree $(d+1)$ and $(d-1)$
binary forms as
\[
\I_\phi := y f_0 - x f_1\; \; \text{ and } \; \; \J_\phi := \frac{\partial f_0}{\partial x} + \frac{\partial f_1}{\partial y}.
\]
Any two degree $d$ rational functions $\phi$ and $\psi$ are conjugate for some $M \in \PGL_2(k)$ via $\psi = \phi^M$ if and only if $\I_{\psi} = \I_{\phi}^M$ and $\J_{\psi} = \J_{\phi}^M$; see \cref{thm-1}. Moreover, there is a one-to-one correspondence between degree $d$ rational functions $\phi(x)$ and points $(f, g) \in V_{d+1} \oplus V_{d-1}$ such that
\[
\Res\left( x g + {\frac{\partial f}{\partial y}}, y g - \frac{\partial f}{\partial x} \right) \neq 0.
\]
Hence, determining invariants of rational functions is the same as determining generators for the ring of invariants of $V_{d+1} \oplus V_{d-1}$, which can be determined using a result of Clebsch once generators of the ring of invariants for $V_{d+1}$ and $V_{d-1}$ are known.

Denote by $\cR_{(d+1, d-1)}$ the ring of invariants of $V_{d+1} \oplus V_{d-1}$ and $(\xi_0, \ldots, \xi_n)$ the tuple of generators of this ring with degrees $(q_0, \ldots, q_n)$ respectively. Since all $\xi_0, \dots, \xi_n$ are homogeneous polynomials, then $\cR_{(d+1, d-1)}$ is a graded ring and $\Proj \cR_{(d+1, d-1)}$ is a weighted projective space denoted by $\P_\w^n(k)$, for some set of weights
where $\w = (q_0, \ldots, q_n)$. 
Thus for each $\phi \in \cP_d$, we evaluate its invariants and have a map $\xi: \Rat_d^1\to \P_{\w}^n$ via
\[
\phi \to [\xi_0(\phi) : \cdots : \xi_n(\phi)].
\]
Determining automorphism groups, describing the inclusions among their loci, and analyzing fields of definition becomes increasingly subtle as the degree $d$ increases; see \cite{silv}. The invariant-theoretic description of $\M_d^1$ grows rapidly in complexity,  and for $d>3$ the structure of the rings $\cR_{(d+1,d-1)}$ is largely unknown.

This setting is closely analogous to the moduli theory of hyperelliptic and superelliptic curves, where explicit invariants play a central role in making the geometry computationally accessible. In the case $d=3$, the invariant theory is fully explicit. For the representation $V_4\oplus V_2$ the ring of invariants is generated by six homogeneous polynomials of degrees $2,2,3,3,4,$ and $6$, giving an embedding
\[
   \M_3^1 \hookrightarrow \P_{(2,2,3,3,4,6)}^{5}
\]
onto the locus where the gcd of the weights of the non-zero coordinates equals $1$ (see \cref{prop-1}). 
We compute these generators and determine the loci corresponding to each finite automorphism group, obtaining equations for the strata with groups $C_2, C_3, C_4, V_4, D_4,$ and $A_4$. As an arithmetic application, we enumerate rational cubics over $\Q$ of naive height $\le 4$ and study their images in the weighted projective moduli space. Working in $\M_3^1$ removes the redundancies inherent in the coefficient space and allows the automorphism strata to be described cleanly and explicitly.

Although the case $d=3$ serves as a complete model, extending this framework to $d>3$ remains an open problem, chiefly because the generators of $\cR_{(d+1,d-1)}$ are unknown. The present work illustrates how an explicit invariant-theoretic description of $\M_d^1$ enables a systematic analysis of automorphism loci and suggests a path
forward for higher degrees.

%***************************************************
\section{Preliminaries}
Let $k$ be an algebraically closed field,  $\P^N(k)$  the projective $N$-space over $k$, and  $k[x, y]$  be the  polynomial ring in  two variables. Throughout this paper we assume $\chara (k)=0$.
By    $V_d$ we denote  the $(d+1)$-dimensional  subspace  of  $k[x, y]$  consisting of homogeneous polynomials   
\[
f(x, y) =    a_d x^d + a_{d-1}x^{d-1} y + \cdots a_1 x y^{d-1} + a_0 y^d, 
\] 
of  degree $d$ (up to multiplication by a scalar).  Elements  in $V_d$  (up to   multiplication by a constant)  are called  \textbf{binary  forms} of degree $d$.    $\GL_2(k)$ acts as a group of automorphisms on $k[x, y] $   as follows:
\begin{equation}
 M =
\begin{bmatrix} a &b \\  c & d
\end{bmatrix}
\in \GL_2(k), \text{   then       }
\quad  M  \begin{bmatrix} x\\ y \end{bmatrix} = \begin{bmatrix} ax+by\\ cx+dy \end{bmatrix}
\end{equation}
Denote by $f^M$ the binary form  $f^M (x, y) := f(ax+by, cx +d y)$.   
It is well known that while $\SL_2(k)$ does not fix a binary form, it preserves its equivalence class under the natural change-of-variables action.
%It is well  known that $\SL_2(k)$ leaves a binary  form (unique up to scalar multiples) on $V_d$ invariant. 

Consider $a_0$, $a_1$,  ... , $a_d$ as parameters  (coordinate  functions on $V_d$). Then the coordinate  ring of $V_d$ can be identified with $ k [a_0 ,  \ldots , a_d] $. For $I \in k [a_0, \ldots , a_d]$ and $M \in \GL_2(k)$, define  $I^M \in k[a_0, \dots , a_d]$  as ${I^M}(f):= I( f^M)$,   for all $f \in V_d$. Then  $I^{MN} = (I^{M})^{N}$ and  ${I^M}(f)$ define an action of $\GL_2(k)$ on $k[a_0, \dots , a_d]$.

A homogeneous polynomial $I\in k[a_0, \dots , a_d, x, y]$ is called a \textbf{covariant}  of index $s$ if   $I^M(f)=\delta^s I(f)$,  for all $f\in V_d$,   where $\delta =\det(M)$.  The homogeneous degree in $a_0, \dots , a_d$ is called the \textbf{degree} of $I$,  and the homogeneous degree in $x, y$ is called the \textbf{order} of $I$.  A covariant of order zero is called \textbf{invariant}.  An invariant is a $\SL_2(k)$-invariant on $V_d$.

By Hilbert's basis theorem  the ring of invariants of binary forms is finitely generated. We denote by $\RR_{d}$ the ring of invariants of the binary forms of degree $d$.   Then,    $\RR_{d}$ is a finitely generated   graded ring; see \cite{2004-3} for details among many other sources. 
 
%------------
\subsection{Change of coordinates}
Let $I_0, \dots  , I_n$ be the generators of $\RR_{d}$ with degrees $q_0, \dots , q_n$ respectively.  For any two binary forms $f$ and $g$,   $f=g^M$,    $M\in \GL_2 (k)$,     if and only if 
\begin{equation}\label{prop-2}
\left( I_0 (f), \dots   I_{i} (f), \dots , I_{n} (f) \right) = \left( \l^{q_0} \, I_0 (g), \dots ,    \l^{q_i}\,  I_i (g), \dots , \l^{q_n} \, I_n (g)   \right), 
\end{equation}
where   $\l = \left( \det M \right)^{\frac d 2}$.

%------------
\subsection{Generators of the ring of invariants}\label{invariants}
Let   $V_d$ be the space of degree $d>1$ binary forms defined over $k$, and    $\cR_d$  the ring of invariants.    Below we list the generating set of $\cR_d$ for $d\leq 10$.   We assume that the  binary forms are  given in standard form 
$f(x, y)= \sum_{i=0}^d \binom{d}{i}  a_i x^i y^{d-i}$.
For   $f, g \in V_d$  the \textbf{$r$-th transvectant} of $f$ and $g$ is defined as
\[(f,g)_r:= \frac {(m-r)! \, (n-r)!} {n! \, m!} \, \,
\sum_{k=0}^r (-1)^k
\begin{bmatrix} r \\ k
\end{bmatrix} \cdot
\frac {\partial^r f} {\partial x^{r-k} \, \,  \partial y^k} \cdot \frac {\partial^r g} {\partial x^k  \, \, \partial y^{r-k} },
\]
While there is no method known to determine a minimal  generating set of invariants for any $\cR_d$,  we display such sets for  $d= 3, 4$ an their degree for $d\leq 10$  as in \cite{curri}.

\begin{exa}[Cubics]   A generating set for $\cR_3$ is   $\xi = [ \xi_0 ]$, where 
 \begin{equation}\label{inv-cubics}
 \xi_0= \left(  (f, f)_2, (f, f)_2    \right)_2 = -54 a_0^2 a_3^2+36 a_1 a_3 a_0 a_2-8 a_2^3 a_0-8 a_1^3 a_3+2 a_2^2 a_1^2
 \end{equation}
 \end{exa}

\begin{exa}[Quartics]   A generating set for $\cR_4$ is   $\xi = [ \xi_0 , \xi_1]$ with $\w=(2,3)$ is the weight of the generating set,   where 
%Let $c_1 = (f, f)_2$.  Then   $J_2 = (f, f)_4$  and $J_3 = (f, c_1)_4$. 
%
\( \xi_0 =   (f, f)_4\)     and \( \xi_1  = \left(   f, (f, f)_2   \right)_4\).
\end{exa}
%

\iffalse
For $ d=6 $ to $ d=10 $, generators are  displayed in \cite{curri}.    In \cref{tab:invariants}    are    their   weights. 
%\subsubsection{Higher Degrees}
%
\begin{table}[h!]
 \caption{Generators of $ \cR_d $ for $ d=6 $ to $ d=10 $}
\begin{center}
\begin{tabular}{|c|c|c|}
\hline
$ d $ & Weights $ \w $ & \# Generators \\
\hline
6 & (2, 4, 6, 10) & 4 \\
7 & (4, 8, 12, 12, 20) & 5 \\
8 & (2, 3, 4, 5, 6, 7) & 6 \\
9 & (4, 8, 10, 12, 12, 14, 16) & 7 \\
10 & (2, 4, 6, 6, 8, 9, 10, 14, 14) & 9 \\
\hline
\end{tabular}
\end{center}
\label{tab:invariants}
\end{table}
%
\fi
  
%*********************   
\section{Invariants of Rational Functions}
This section derives the invariants of degree $d > 1$ rational functions on $\P^1(k)$, where $k$ is algebraically closed. Our goal is to classify these functions up to $\PGL_2(k)$-equivalence, compute the invariant ring $\cR_{(d+1, d-1)}$, and specialize to $d=3$ for explicit invariants that enable machine learning analysis of the moduli space $\M_3^1$. 

If one fixes homogeneous coordinates $x, y$ on $\P^1$, then any rational function $\phi: \P^1 \to \P^1$ of degree $d > 1$ can be realized as
\begin{equation}
\phi(x, y) = \frac{f_0(x, y)}{f_1(x, y)},
\end{equation}
where $d = \max \{\deg f_0, \deg f_1\}$ and $f_0$ and $f_1$ are homogeneous polynomials of degree at most $d$.
Conversely, a pair of homogeneous polynomials of degree $d$ in $x$ and $y$ determine a rational function
\begin{equation}\label{phi}
\phi(x, y) = [f_0(x, y) : f_1(x, y)]
\end{equation}
if $f_0, f_1$ have no common roots in $k$.
For a fixed degree $d > 1$, the collection of all such pairs of homogeneous polynomials $[f_0(x, y) : f_1(x, y)]$, say
\begin{equation}
f_0 = \sum_{i=0}^d a_i x^{d-i} y^i 
\quad 
\text{  and  } 
\quad 
f_1 = \sum_{i=0}^d b_i x^{d-i} y^i,
\end{equation}
can be naturally parametrized as the projective space $\P^{2d+1}$, via
\[
[f_0 : f_1] \to [a_0 : a_1 : \dots : a_d : b_0 : \dots : b_d] \in \P^{2d+1}.
\]
We denote the resultant of two binary forms $f_0$ and $f_1$ by $\Res(f_0, f_1)$. Notice that it’s well-defined and a degree $2d$ homogeneous polynomial in 
\begin{equation}\label{I-2d}
I_{2d}(\phi) := \Res(f_0, f_1) \in k[a_0, \ldots, a_d, b_0, \ldots, b_d].
\end{equation}
Hence,  $I_{2d}(\phi)$ is an $\SL_2(k)$-invariant of degree $2d$. Moreover, $\phi$ is a rational function on $\P^1$ if and only if $I_{2d}(\phi) \neq 0$.

We can construct the parameter space of rational functions on $\P^1$ as the complement of the vanishing locus $V(I_{2d})$ of $I_{2d}$. Hence the \textbf{space of rational functions of degree $d$ on $\P^1$} is defined as
\[
\Rat_d^1 := \P^{2d+1} \setminus V(I_{2d}).
\]
The action of $\PGL_2(k)$ on $V_d$ extends naturally to an action on $\Rat_d^1$. For each $\sigma \in \PGL_2(k)$, we have $\PGL_2(k) \times \Rat_d^1 \to \Rat_d^1$ via
\[
\begin{split}
\left( \sigma, \phi(x, y) \right) & \to \phi^\sigma := \sigma^{-1} \phi \sigma,
\end{split}
\]
Two rational functions $\phi, \psi \in \Rat_d^1$ are called $k$-\textbf{conjugate} if and only if there exists a matrix $\sigma \in \PGL_2(k)$ such that 
\[
\psi = \sigma^{-1}  \phi  \sigma.
\]

\begin{rem}
Notice two different uses of notation $\phi^\sigma$ for rational functions and $f^\sigma$ for binary forms.
\end{rem}

The function $\phi^\sigma = \sigma^{-1} \phi \sigma$ is explicitly given as
\begin{equation}\label{phi-M}
\begin{split}
\left( \sigma^{-1} \phi \sigma \right)(x) &= \frac{e \, f_0(a x + b, c x + e) - b \, f_1(a x + b, c x + e)}{-c \, f_0(a x + b, c x + e) + a \, f_1(a x + b, c x + e)} 
= \frac{e \, f_0^\sigma - b \, f_1^\sigma}{-c \, f_0^\sigma + a \, f_1^\sigma}
\end{split}
\end{equation}
Let $\phi(x, y)$ and $\psi(x, y)$ be degree $d \geq 2$ rational functions given by
\begin{equation}
\phi(x, y) = \frac{f_0(x, y)}{f_1(x, y)} \quad \text{and} \quad \psi(x, y) = \frac{g_0(x, y)}{g_1(x, y)}.
\end{equation}
By \cref{phi-M}, $\phi$ and $\psi$ are $k$-conjugate if and only if there is $\sigma = \begin{bmatrix} a & b \\ c & e \end{bmatrix} \in \PGL_2(k)$ such that
\begin{equation}\label{eq-3}
g_0 = e f_0^\sigma - b f_1^\sigma \quad \text{and} \quad g_1 = -c f_0^\sigma + a f_1^\sigma.
\end{equation}
\begin{defi}
For $\phi(x, y) = \frac{f_0(x, y)}{f_1(x, y)}$, define its \textbf{associated pair of forms} as
\begin{equation}\label{F-G}
\I_\phi := y f_0 - x f_1 \in V_{d+1}   \quad \text{and} \quad \J_\phi := \frac{\partial f_0}{\partial x} + \frac{\partial f_1}{\partial y}  \in V_{d-1}
\end{equation}
\end{defi}

\begin{lem}\label{thm-1}
Let $\phi, \psi \in \Rat_d^1$ and $\sigma \in \PGL_2(k)$. Then $\psi = \phi^\sigma$ if and only if
\[
\I_\psi = \I_\phi^\sigma \quad \text{   and } \quad \J_\psi = \J_\phi^\sigma.
\]
\end{lem}

\begin{proof}
Let $\phi = \frac{f_0}{f_1}$ and $\psi = \frac{g_0}{g_1}$. Assume that $\phi$ and $\psi$ are $k$-conjugate in $\Rat_d^1$, meaning there exists $\sigma \in \PGL_2(k)$ such that $\psi = \phi^\sigma$.

Substituting the expressions for $\I_\psi = y g_0 - x g_1$ and using the values of $g_0$ and $g_1$ as in \cref{eq-3}, we have
\[
\begin{split}
\I_\psi &= y g_0 - x g_1 = y (e f_0^\sigma - b f_1^\sigma) - x (-c f_0^\sigma + a f_1^\sigma) \\
&= (c x + e y) f_0^\sigma - (a x + b y) f_1^\sigma = \left( y f_0 - x f_1 \right)^\sigma = \I_\phi^\sigma.
\end{split}
\]
Similarly,
\[
\begin{split}
\J_\psi &= \frac{\partial g_0}{\partial x} + \frac{\partial g_1}{\partial y} = \frac{\partial (e f_0^\sigma - b f_1^\sigma)}{\partial x} + \frac{\partial (-c f_0^\sigma + a f_1^\sigma)}{\partial y} \\
&= e \frac{\partial f_0^\sigma}{\partial x} - b \frac{\partial f_1^\sigma}{\partial x} - c \frac{\partial f_0^\sigma}{\partial y} + a \frac{\partial f_1^\sigma}{\partial y} 
= \left( \frac{\partial f_0^\sigma}{\partial x} + \frac{\partial f_1^\sigma}{\partial y} \right)^\sigma = \J_\phi^\sigma.
\end{split}
\]
Thus, we conclude that $\I_\phi$ and $\I_\psi$ (and similarly, $\J_\phi$ and $\J_\psi$) are $k$-equivalent via $\sigma$ as binary forms.

Conversely, suppose that $\I_\phi$ and $\I_\psi$ (respectively, $\J_\phi$ and $\J_\psi$) are $k$-equivalent via $\sigma$ as binary forms. This means that $\I_\psi = \I_\phi^\sigma$ and $\J_\psi = \J_\phi^\sigma$. In particular, we have
\[
\begin{split}
\I_\psi &= y g_0 - x g_1 = \left( y f_0 - x f_1 \right)^\sigma = (c x + e y) f_0^\sigma - (a x + b y) f_1^\sigma, \\
\J_\psi &= \frac{\partial g_0}{\partial x} + \frac{\partial g_1}{\partial y} = \frac{\partial f_0^\sigma}{\partial x^\sigma} + \frac{\partial f_1^\sigma}{\partial y^\sigma}.
\end{split}
\]
From the equation for $\I_\psi$, we obtain
\[
y \left( g_0 - e f_0^\sigma + b f_1^\sigma \right) = x \left( g_1 + c f_0^\sigma - a f_1^\sigma \right),
\]
which leads to
$g_0 - e f_0^\sigma + b f_1^\sigma = x \cdot h(x, y)$,  and 
$g_1 + c f_0^\sigma - a f_1^\sigma  = y \cdot h(x, y)$, 
for some $h \in V_{d-1}$. Therefore, we have
\[
\begin{split}
\J_\psi &= 2 h + \left( x \frac{\partial h}{\partial x} + y \frac{\partial h}{\partial y} \right) + e \frac{\partial f_0^\sigma}{\partial x} + b \frac{\partial f_1^\sigma}{\partial x} - c \frac{\partial f_0^\sigma}{\partial y} + a \frac{\partial f_1^\sigma}{\partial y} \\
&= h + \frac{d-1}{2} h + \frac{\partial g_0}{\partial x} + \frac{\partial g_1}{\partial y}.
\end{split}
\]
Thus, we must have $h = 0$, which implies that $\phi^\sigma = \psi$ as claimed.
\end{proof}

Since the pair of binary forms $(\I_\phi, \J_\phi)$ determines the rational function $\phi$, we can use the classical theory of binary forms to determine invariants for $\phi$. Define
\begin{equation}\label{eq:Phi}
\Phi : \Rat_d^1  \to V_{d+1} \oplus V_{d-1}, 
\end{equation}
via $\Phi( \phi ) =  \left( \I_\phi, \J_\phi \right)$.
The inverse of $\Phi$ is not well-defined since not every pair $(f, g) \in V_{d+1} \oplus V_{d-1}$ determines a rational function.
\begin{defi}
For any $(f, g) \in V_{d+1} \oplus V_{d-1}$, define the \textbf{modular resultant}
\begin{equation}
\Delta_{(f,g)} = \Res \left( x g + \frac{\partial f}{\partial y}, \, y g - \frac{\partial f}{\partial x} \right),
\end{equation}
and the \textbf{moduli resultant locus} $\cN$ as
\begin{equation}\label{cN}
\cN := \{ (f, g) \in V_{d+1} \oplus V_{d-1} \mid \Delta_{(f,g)} = 0 \}.
\end{equation}
\end{defi}

Then, we have the following  result which has also appeared in \cite{west}. 
\begin{lem}
The map 
$
\Phi : \Rat_d^1 \to (V_{d+1} \oplus V_{d-1}) \setminus \cN
$
 is bijective. Moreover, for any $(f, g) \in V_{d+1} \oplus V_{d-1} \setminus \cN$,
\[
\Phi^{-1}(f, g) = \frac{x g + \frac{\partial f}{\partial y}}{y g - \frac{\partial f}{\partial x}}.
\]
\end{lem}

\begin{proof}
The map $\Phi$ is  well-defined. Indeed, for $\varphi = f_{0}/f_{1}$ one verifies directly that
\[
x J_\varphi + \frac{\partial I_\varphi}{\partial y} = d\, f_{0}, \qquad
y J_\varphi - \frac{\partial I_\varphi}{\partial x} = d\, f_{1},
\]
and hence
$
\Delta_{I_\varphi, J_\varphi}
= d^{\,2d}\, \operatorname{Res}(f_{0}, f_{1}) \neq 0,
$
so $\Phi(\varphi)\notin N$.

Let $\phi, \psi \in \Rat_d^1$ such that $\Phi(\phi) = \Phi(\psi)$. Then, $\I_\phi = \I_\psi$ and $\J_\phi = \J_\psi$ as binary forms in $V_{d+1}$ and $V_{d-1}$, respectively. Since binary forms are defined up to multiplication by a scalar, there exists a diagonal matrix $\sigma \in \PGL_2(k)$ such that  $\I_\psi = \I_\phi^\sigma$   and $ \J_\psi = \J_\phi^\sigma$.
By \cref{thm-1}, we conclude that $\psi = \phi^\sigma = \phi$. Thus, $\Phi$ is injective.

Now, assume that $(f, g) \in (V_{d+1} \oplus V_{d-1}) \setminus \cN$. The condition $\Delta_{(f,g)} \neq 0$ ensures that the preimage
\[
\Phi^{-1}(f, g) = \frac{x g + \frac{\partial f}{\partial y}}{y g - \frac{\partial f}{\partial x}}
\]
belongs to $\Rat_d^1$. Consequently, the map $\Phi^{-1}$ is well-defined.

To show that $\Phi$ and $\Phi^{-1}$ are inverses of each other, we will demonstrate that $\Phi \circ \Phi^{-1} = \text{id}$ on $(V_{d+1} \oplus V_{d-1}) \setminus \cN$ and $\Phi^{-1} \circ \Phi = \text{id}$ on $\Rat_d^1$.

First, compute $\Phi \circ \Phi^{-1}$ on $(V_{d+1} \oplus V_{d-1}) \setminus \cN$: Let $(f, g) \in (V_{d+1} \oplus V_{d-1}) \setminus \cN$ and  $\phi = \Phi^{-1}(f, g) = \frac{x g + \frac{\partial f}{\partial y}}{y g - \frac{\partial f}{\partial x}}$. Then,
\[
\begin{split}
\I_\phi &= y \left( x g + \frac{\partial f}{\partial y} \right) - x \left( y g - \frac{\partial f}{\partial x} \right) = y \frac{\partial f}{\partial y} + x \frac{\partial f}{\partial x} = (d + 1) f    \in  V_{d+1} \\   
\J_\phi &= \frac{\partial}{\partial x} \left[ x g + \frac{\partial f}{\partial y} \right] + \frac{\partial}{\partial y} \left[ y g - \frac{\partial f}{\partial x} \right] = 2 g + x \frac{\partial g}{\partial x} + y \frac{\partial g}{\partial y} = 2 g + (d - 1) g   \in V_{d-1} 
\end{split}
\]
Here $(d+1)f$ and $(d+1)g$ are nonzero scalar multiples of $f$ and $g$, and thus represent the same binary forms in $V_{d+1}$ and $V_{d-1}$.
Thus, $\I_\phi = f$ and $\J_\phi = g$, so $\Phi(\phi) = (\I_\phi, \J_\phi) = (f, g)$.

Second, compute $\Phi^{-1} \circ \Phi$ on $\Rat_d^1$: Let $\phi = \frac{f_0}{f_1} \in \Rat_d^1$. Then,
\[
\begin{split}
\left(\Phi^{-1} \circ \Phi\right)(\phi) &= \Phi^{-1}(\I_\phi, \J_\phi) = \Phi^{-1}\left(y f_0 - x f_1, \frac{\partial f_0}{\partial x} + \frac{\partial f_1}{\partial y}\right) \\
&= \frac{x \left( \frac{\partial f_0}{\partial x} + \frac{\partial f_1}{\partial y} \right) + \frac{\partial}{\partial y} \left[ y f_0 - x f_1 \right]}{y \left( \frac{\partial f_0}{\partial x} + \frac{\partial f_1}{\partial y} \right) - \frac{\partial}{\partial x} \left[ y f_0 - x f_1 \right]} \\
&= \frac{x \left( \frac{\partial f_0}{\partial x} + \frac{\partial f_1}{\partial y} \right) + \left( f_0 + y \frac{\partial f_0}{\partial y} - x \frac{\partial f_1}{\partial y} \right)}{y \left( \frac{\partial f_0}{\partial x} + \frac{\partial f_1}{\partial y} \right) - \left( y \frac{\partial f_0}{\partial x} - f_1 - x \frac{\partial f_1}{\partial x} \right)} \\
&= \frac{x \frac{\partial f_0}{\partial x} + y \frac{\partial f_0}{\partial y} + f_0}{x \frac{\partial f_1}{\partial x} + y \frac{\partial f_1}{\partial y} + f_1} = \frac{d f_0 + f_0}{d f_1 + f_1} = \frac{f_0}{f_1} = \phi.
\end{split}
\]
Since both compositions $\Phi \circ \Phi^{-1}$ and $\Phi^{-1} \circ \Phi$ are the identity maps, we conclude that $\Phi$ and $\Phi^{-1}$ are indeed inverse to each other. This completes the proof.
\end{proof}

%***************************************
\subsection{Ring of invariants $ \cR_{(d+1), (d-1)}$  and a theorem of Clebsch}
The action of $\GL_2(k)$ on $V_d$ induces an action of $\GL_2(k)$ in $V_{d+1} \oplus V_{d-1}$. To determine the isomorphism classes of degree $d$ rational functions, we have to determine the ring of invariants of $V_{d+1} \oplus V_{d-1}$. This is a well known in classical invariant theory.   We briefly describe it below;  see \cite{west} for details. 

Let $V$ be an $\SL_2$-module and $\mathcal{O}(V)$ the algebra of polynomial functions on $V$. $\SL_2(k)$ acts on $\mathcal{O}(V)$ via
\[
M \cdot p(f_1, \ldots, f_r) \to p(M^{-1} f_1, \ldots, M^{-1} f_r),
\]
for every $M \in \SL_2(k)$. An invariant of $V$ is an element $\cT \in \mathcal{O}(V)$ such that $M \cT = \cT$, for all $M \in \SL_2(k)$. The set of invariants is denoted by $\mathcal{O}(V)^{\SL_2}$.

A transvectant $(\cT, \cS)_l$ is called \textbf{irrelevant} if there exist $\cT_1, \cT_2, \cS_1, \cS_2$ and $l_1, l_2$ such that
\[
l = l_1 + l_2, \quad \cT = \cT_1 \cdot \cT_2, \quad \cS = \cS_1 \cdot \cS_2,
\]
and $l_1 \leq \ord \cT_1, \ord \cS_1$, and $l_2 \leq \ord \cT_2, \ord \cS_2$. A transvectant which is not irrelevant is called \textbf{relevant}.

Let $V$ and $W$ be two $\SL_2$-modules whose covariants are finitely generated, and assume
\begin{equation}
\begin{split}
& \cT_1, \ldots, \cT_r: \quad \text{are the generators of the covariants of } V \\
& \cS_1, \ldots, \cS_s: \quad \text{are the generators of the covariants of } W.
\end{split}
\end{equation}

\begin{thm*}[Clebsch]\label{Clebsch}
The ring of covariants of $V \oplus W$ is  finitely generated. Moreover, a finite generating system can be chosen from the set of all transvectants
\[
\left( \cT, \cS \right)_l, \quad \text{for} \quad l \geq 0,
\]
where $\cT$ is a monomial in the $\cT_i$'s and $\cS$ a monomial in the $\cS_j$'s. In other words, by the relevant transvectants $(\cT, \cS)_l$.
\end{thm*}

%***************************************************************
\subsection{ $\Proj \cR_{(d+1, d-1)}$  as a weighted projective space}
Let $\xi_0, \dots, \xi_n$ be a generating system of $\cR_{(d+1, d-1)}$. Since all $\xi_0, \dots, \xi_n$ are homogeneous polynomials, then $\cR_{(d+1, d-1)}$ is a graded ring and $\Proj \cR_{(d+1, d-1)}$ is a weighted projective space.

Let $\w := (q_0, \dots, q_n) \in \Z^{n+1}$ be a fixed tuple of positive integers called \textbf{weights}. Consider the action of $k^\times = k \setminus \{0\}$ on $\A^{n+1}(k)$ as follows:
\[
\lambda \star (x_0, \dots, x_n) = \left( \lambda^{q_0} x_0, \dots, \lambda^{q_n} x_n \right)
\]
for $\lambda \in k^\times$. The quotient of this action is called a \textbf{weighted projective space} and denoted by $\P^n_\w(k)$. It is the projective variety $\Proj(k[x_0, \dots, x_n])$ associated to the graded ring $k[x_0, \dots, x_n]$ where the variable $x_i$ has degree $q_i$ for $i = 0, \dots, n$.

We will denote a point $\p \in \P^n_\w(k)$ by $\p = [x_0 : x_1 : \dots : x_n]$.    Let $\phi(x, y) \in \Rat_d^1$ given by
$\phi(x, y) = [f_0: f_1]$.
Its associated binary forms $\I_\phi \in V_{d+1}$ and $\J_\phi \in V_{d-1}$, and $\xi_0, \xi_1, \dots, \xi_n$ the generators of the ring of invariants $\cR_{(d+1, d-1)}$.

The \textbf{invariants} of the rational function $\phi$ are defined as
\begin{equation}\label{xi}
\xi(\phi) := \left[ \xi_0(\I_\phi, \J_\phi), \xi_1(\I_\phi, \J_\phi), \dots, \xi_n(\I_\phi, \J_\phi) \right] \in \P^n_\w(k).
\end{equation}
Moreover, $\phi = \psi^\sigma$ for $\sigma \in \GL_2(k)$ if and only if $\xi(\phi) = \lambda \star \xi(\psi)$, for $\lambda = (\det \sigma)^{\frac{d}{2}}$.

Next we will determine explicitly invariants of $\cR_{d+1, d-1}$. From now on $f \in V_{d+1}$ and $g \in V_{d-1}$ where
\begin{equation}\label{f-g}
f = \sum_{i=0}^{d+1} a_i x^i y^{d+1-i} \quad \text{and} \quad g = \sum_{i=0}^{d-1} b_i x^i y^{d-1-i}.
\end{equation}

%*************************************
\subsection{Invariants of $V_4 \oplus V_2$}
To apply this framework, we specialize to $d=3$, where $\I_\phi \in V_4$ and $\J_\phi \in V_2$. We take $d=3$, $f \in V_4$ and $g \in V_2$ as in \cref{f-g}
\begin{equation}
\begin{split}
f(x, y) &= a_4 x^4+ a_3 x^3y + a_2 x^2 y^2 + a_1 x y^3 + a_0 y^4 \\
g(x, y) &= b_2 x^2 + b_1 x y + b_0 y^2
\end{split}
\end{equation}
The generators of covariants of $V_4$ and $V_2$ are
\[
T = \{ f, \, \cT = (f, f)_2, \, \cT_2 = (f, f)_4, \, \cT_3 = ((f, f)_2, f)_4 \} \quad \text{ and } \quad
S = \{ g, \, \cS_2 = (g, g)_2 \}
\]
respectively.  Hence, we are considering all transvectants
\[
\left( f^{m_1} \cT^{m_2} \cT_2^{m_3} \cT_3^{m_4}, g^{s_1} S_2^{s_2} \right)_l,
\]
for some $m_1$, $m_2$, $m_3$, $m_4$, $s_1$, $s_2$.  Since $\cT_2, \cT_3$ and $S_2$ are invariants, their exponents must be zero, otherwise we get reducible invariants.
Hence, $\cT_2, \cT_3, S_2$ are part of the generating set and further we only consider $\left( f^{m_1} \cT^{m_2}, g^{s} \right)_l$.  Then the relevant transvectants are $S_2, \cT_2, \cT_3$, and
$R_3 := (\cT, g^2)_4$, $R_4 := (f, g^2)_4$, $R_6 := (g^3, (f, \cT)_1)_6$.
Hence, the set of invariants is $\xi(\phi) = \left( \xi_0, \ldots, \xi_5 \right)$, where
\[
\xi_0 = (g, g)_2, \; \xi_1 = (f, f)_4, \; \xi_2 = (\cT, f)_4, \; \xi_3 = R_3, \; \xi_4 = R_4, \; \xi_5 = R_6
\]
with weights $(2, 2, 3, 3, 4, 6)$ respectively. Let $d=3$ and $f, g$ as in \cref{f-g}. 
We have the following expressions for invariants:
\[
\begin{split}
\xi_0 &= \frac{1}{2} \left( 4 b_0 b_2 - b_1^2 \right) \\
\xi_1 &= \frac{1}{6} \left( a_2^2 - 3 a_1 a_3 + 12 a_0 a_4 \right) \\
\xi_2 &= \frac{1}{72} \left( -2 a_2^3 + 9 (a_1 a_3 + 8 a_0 a_4) a_2 - 27 \left( a_4 a_1^2 + a_0 a_3^2 \right) \right) \\
\xi_3 &= \frac{1}{6} \left( 6 a_4 b_0^2 - 3 a_3 b_1 b_0 + 2 a_2 b_2 b_0 + a_2 b_1^2 + 6 a_0 b_2^2 - 3 a_1 b_1 b_2 \right) \\
\xi_4 &= -\frac{1}{72} \left( 2 a_2^2 b_1^2 + 4 a_2^2 b_0 b_2 - 24 a_4 a_2 b_0^2 - 24 a_0 a_2 b_2^2 - 6 a_3 a_2 b_0 b_1 \right. \\
&\left. -6 a_1 a_2 b_1 b_2 + 9 a_3^2 b_0^2 - 3 a_1 a_3 b_1^2 - 24 a_0 a_4 b_1^2 + 9 a_1^2 b_2^2 + 36 a_1 a_4 b_0 b_1 - 6 a_1 a_3 b_0 b_2 \right. \\
&\left. -48 a_0 a_4 b_0 b_2 + 36 a_0 a_3 b_1 b_2 \right) \\
\xi_5 &= -\frac{1}{32} \left( a_3^3 b_0^3 + 8 a_1 a_4^2 b_0^3 - 4 a_2 a_3 a_4 b_0^3 - a_2 a_3^2 b_1 b_0^2 -8 a_0 a_1 a_4 b_1^2 b_2 - 16 a_0 a_4^2 b_1 b_0^2 \right. \\
&\left. -2 a_1 a_3 a_4 b_1 b_0^2 + a_1 a_3^2 b_2 b_0^2 - 4 a_1 a_2 a_4 b_2 b_0^2 + 8 a_0 a_3 a_4 b_2 b_0^2 + a_1 a_3^2 b_1^2 b_0 + 4 a_2^2 a_4 b_1 b_0^2 \right. \\
&\left. -4 a_1 a_2 a_4 b_1^2 b_0 + 8 a_0 a_3 a_4 b_1^2 b_0 - a_1^2 a_3 b_2^2 b_0 + 4 a_0 a_2 a_3 b_2^2 b_0 - 8 a_0 a_1 a_4 b_2^2 b_0 \right. \\
&\left. -6 a_0 a_3^2 b_1 b_2 b_0 + 6 a_1^2 a_4 b_1 b_2 b_0 - a_0 a_3^2 b_1^3 + a_1^2 a_4 b_1^3 - a_1^3 b_2^3 + 4 a_0 a_1 a_2 b_2^3 - 8 a_0^2 a_3 b_2^3 \right. \\
&\left. -4 a_0 a_2^2 b_1 b_2^2 + a_1^2 a_2 b_1 b_2^2 + 2 a_0 a_1 a_3 b_1 b_2^2 + 16 a_0^2 a_4 b_1 b_2^2 - a_1^2 a_3 b_1^2 b_2 + 4 a_0 a_2 a_3 b_1^2 b_2 \right) \\
\end{split}
\]
The ring of invariants $\cR_{4,2}$ is generated by $\xi_0, \xi_1, \xi_2, \xi_3, \xi_4, \xi_5$ and a relation between invariants $\xi_0, \ldots, \xi_5$ is, according to \cite{west}, satisfy the following equation
\begin{small}
\begin{equation}\label{eq:Rat}
\xi_5^2 = \frac{1}{108} \xi_0^3 \xi_1^3 - 18 \xi_0^3 \xi_2^2 - \frac{1}{24} \xi_0 \xi_1^2 \xi_3^2 - \frac{1}{6} \xi_2 \xi_3^3 + \frac{1}{2} \xi_0 \xi_2 \xi_3 \xi_4 + \frac{1}{4} \xi_1 \xi_3^2 \xi_4 - \frac{1}{4} \xi_0 \xi_1 \xi_4^2 - \frac{1}{2} \xi_4^3
\end{equation}
\end{small}
Given a rational cubic
\begin{equation}\label{phi-deg3}
\phi(x) = \frac{f_0(x)}{f_1(x)} = \frac{c_0 x^3 + c_1 x^2 + c_2 x + c_3}{c_4 x^3 + c_5 x^2 + c_6 x + c_7}
\end{equation}
we   compute   its invariants in terms of its coefficients. 
%Hence, $\phi$ is represented by a point $\p = [c_0, c_1, \ldots, c_7]$ in the projective space $\P^7$. 
We have   $I_6 = \Res(f_0, f_1)$ .
%
%\iffalse
\[
\begin{split}
I_6 = & c_3^3 c_4^3 - c_0^3 c_7^3 + c_3 c_0^2 c_6^3 - c_2^3 c_7 c_4^2 + c_1^3 c_7^2 c_4 - c_3^2 c_2 c_5 c_4^2 - c_3^2 c_0 c_5^3 - 2 c_3^2 c_1 c_6 c_4^2  \\
& + c_3 c_2^2 c_6 c_4^2 + c_3^2 c_1 c_5^2 c_4 - 3 c_3^2 c_0 c_7 c_4^2 - c_1^2 c_0 c_7^2 c_5 + c_1 c_0^2 c_7^2 c_6 + 2 c_2 c_0^2 c_7^2 c_5  \\
&  + 3 c_3 c_0^2 c_7^2 c_4 - c_2^2 c_0 c_7 c_5^2 + 3 c_3^2 c_0 c_6 c_5 c_4 - 2 c_3 c_2 c_0 c_6^2 c_4 + c_3 c_2 c_0 c_6 c_5^2 \\
&  + 3 c_3 c_2 c_1 c_7 c_4^2 + 2 c_3 c_1 c_0 c_7 c_5^2 - c_3 c_1 c_0 c_6^2 c_5 - c_2 c_1^2 c_7 c_6 c_4 - 3 c_2 c_1 c_0 c_7^2 c_4  \\
&  + c_2^2 c_1 c_7 c_5 c_4 + 2 c_2^2 c_0 c_7 c_6 c_4 - c_3 c_2 c_1 c_6 c_5 c_4 - c_3 c_2 c_0 c_7 c_5 c_4 + c_3 c_1 c_0 c_7 c_6 c_4   \\
& + c_3 c_1^2 c_6^2 c_4  - c_2 c_0^2 c_7 c_6^2 + c_2 c_1 c_0 c_7 c_6 c_5  - 3 c_3 c_0^2 c_7 c_6 c_5 - 2 c_3 c_1^2 c_7 c_5 c_4 
\end{split}
\]
%\fi

The pair of binary forms $\I_\phi$ and $\J_\phi$ associated to $\phi$ are
\begin{equation}
\begin{split}
\I_\phi &= c_3 x^3 y + c_2 x^2 y^2 + c_1 x y^3 + c_0 y^4 - c_7 x^4 - c_6 x^3 y - c_5 x^2 y^2 - c_4 x y^3 \\
\J_\phi &= 3 c_3 x^2 + 2 c_2 x y + c_1 y^2 + c_6 x^2 + 2 c_5 x y + 3 c_4 y^2
\end{split}
\end{equation}
Next we evaluate the following transvectants
\begin{equation}
\begin{aligned}
\xi_0 &= \left( \J_\phi, \J_\phi \right)_2, & \xi_1 &= \left( \I_\phi, \I_\phi \right)_4, & \xi_2 &= \left( (\I_\phi, \I_\phi)_2, \I_\phi \right)_4, \\
\xi_3 &= \left( \I_\phi, \J_\phi^2 \right)_4, & \xi_4 &= \left( (\I_\phi, \I_\phi)_2, \J_\phi^2 \right)_4, & \xi_5 &= \left( \J_\phi^3, \left( \I_\phi, (\I_\phi, \I_\phi)_2 \right)_1 \right)_6
\end{aligned}
\end{equation}

\begin{rem}
The full symbolic expressions for the invariants $\xi_0,\ldots,\xi_5$ and the
auxiliary identities used in their computation follow directly from the formulas
given above.   We display them in  \cref{app-a}. 
\end{rem}

The modular resultant $\Delta_{\I_\phi, \J_\phi} = \Res(\I_\phi, \J_\phi)$ is a homogeneous polynomial of degree six in terms of $c_0, \ldots, c_7$.
Hence, we will denote it by 
$J_6 := \Res(\I_\phi, \J_\phi)$. 

Thus, there are three invariants of degree 6 in the case of cubics, namely $I_6$, $J_6$, and $\xi_5$. We would like to express $I_6$ and $J_6$ in terms of $\xi_0, \ldots, \xi_5$. The expression of $I_6$ was computed in \cite[pg. 38]{west}
\begin{equation}\label{eq:I6-old}
I_6 = \frac{1}{8} \xi_1^3 + \frac{1}{384} \xi_1 \xi_0^2 - \frac{3}{4} \xi_2^2 - \frac{3}{16} \xi_2 \xi_4 + \frac{1}{256} \xi_4^2 + \frac{3}{16} \xi_1 \xi_3 - \frac{1}{64} \xi_0 \xi_3 - \frac{1}{8} \xi_5
\end{equation}
It seems there are a couple of typos in the printed version of \cite[pg. 38]{west}, and we could not verify it directly. It can be easily noticeable that it is incorrect since it is not a homogeneous polynomial of degree 6; see for example the monomials $\xi_4^2$, $\xi_1, \xi_3$, $\xi_0 \xi_3$ which are not of degree 6.

The correct expression can be derived using  computational algebra. This involves expressing $I_6$ and $J_6$ as linear combinations of the 10 degree 6 invariants: $\xi_5$, $\xi_4 \xi_0$, $\xi_4 \xi_1$, $\xi_2^{2-j} \xi_3^j$, $\xi_0^{3-i} \xi_1^i$, for $i=0,1,2,3$ and $j=0,1,2$. We have
\[
\begin{split}
I_6 &= -\frac{1}{8} \xi_1^3 - \frac{1}{384} \xi_0^2 \xi_1 + \frac{3}{4} \xi_2^2 - \frac{3}{16} \xi_1 \xi_4 - \frac{1}{256} \xi_3^2 + \frac{3}{16} \xi_2 \xi_3 + \frac{1}{64} \xi_0 \xi_4 - \frac{1}{8} \xi_5 \\
J_6 &= \xi_3^2 - 4 \xi_4 \xi_0 + \frac{2}{3} \xi_0^2 \xi_1
\end{split}
\]
%
%------------------------- Here 
\subsection{A natural map from weighted projective space to projective space}
Let $\P(q_0,\dots,q_n)=\Proj k[x_0,\dots,x_n]$, $\deg(x_i)=q_i$, and $d=\lcm(q_0,\dots,q_n)$.   Then the monomials $x_0^{d/q_0}$, $x_1^{d/q_1}$, $\dots $, $x_n^{d/q_n}$ all have weighted degree \(d\). Hence they define a morphism $\phi_d:\P(q_0,\dots,q_n)\longrightarrow \P^n$ given by
\[
[x_0:\cdots:x_n]   \longmapsto    [x_0^{d/q_0}:\cdots:x_n^{d/q_n}].
\]
The following is a well known fact on weighted projective spaces. We include its proof due to lack of a precise reference.
\begin{lem}
The morphism $\phi_d:\P(q_0,\dots,q_n)\to\P^n $ is finite. The locus where $\phi_d$ fails to be injective is exactly the union of coordinate strata $x_{i_1}=\cdots=x_{i_r}=0$
for which $\gcd(q_j: j\notin\{i_1,\dots,i_r\})>1$.
\end{lem}

\begin{proof}
The map $\phi_d$ is induced by the $\G_m$-equivariant morphism
\[
F:\A^{n+1}\setminus\{0\}\longrightarrow \A^{n+1}\setminus\{0\},
\qquad
(x_0,\dots,x_n)\longmapsto    (x_0^{d/q_0},\dots,x_n^{d/q_n}),
\]
since each coordinate has weight $d$.
Suppose two points represent the same image under $\phi_d$. Then there exists  $t\in k^\ast$ such that
\[
x_i'^{\,d/q_i}=t\,x_i^{d/q_i}  \qquad (i=0,\dots,n).
\]
Thus $x_i'= \zeta_i\, t^{q_i/d}x_i$, where $\zeta_i$ is a $(d/q_i)$-th root of unity.   For the points to represent the same point of $\P(q_0,\dots,q_n)$ there must exist
$\lambda\in k^\ast$ such that $x_i'=\lambda^{q_i}x_i$.
Comparing the two expressions shows that the ambiguity is exactly given by the roots of unity $\mu_g$, where $g=\gcd(q_i: x_i\neq0)$.  Therefore the map is injective precisely when $g=1$, and it fails to be injective exactly along the coordinate strata where the weights of the nonzero coordinates have a common divisor greater than one.
\end{proof}

Now we turn our attention to our space $\P_{(2,2,3,3,4,6)}$.   We consider the Veronese-type morphism
\begin{equation}\label{veronese}
[\xi_0, \xi_1, \xi_2, \xi_3, \xi_4, I_6] \to [\xi_0^6, \xi_1^6, \xi_2^4, \xi_3^4, \xi_4^3, I_6^2]
\end{equation}
Since $I_6 \neq 0$, we can divide by $I_6^2$ and represent each point as
\[
\left[ \frac{\xi_0^6}{I_6^2} : \frac{\xi_1^6}{I_6^2} : \frac{\xi_2^4}{I_6^2} : \frac{\xi_3^4}{I_6^2} : \frac{\xi_4^3}{I_6^2} : 1 \right] \in \P^5.
\]
The map fails to be injective along the strata where the weights of the nonzero coordinates have a common divisor (i.e., where the weighted $\G_m$-action has nontrivial stabilizer).  This motivates the definition of the following invariants
\begin{equation}\label{absinv}
i_1 = \frac{\xi_0^6}{I_6^2}, \; i_2 = \frac{\xi_1^6}{I_6^2}, \; i_3 = \frac{\xi_2^4}{I_6^2}, \; i_4 = \frac{\xi_3^4}{I_6^2}, \; i_5 = \frac{\xi_4^3}{I_6^2},
\end{equation}
which are $\GL_2(k)$-invariants and are defined everywhere in the moduli space $\Rat_3^1$. We call such invariants $i_1, \ldots, i_5$ \textbf{absolute invariants} of $\phi(x, y)$. Hence, we have:

\begin{prop}\label{prop-1}
Two degree-three rational functions are conjugate if and only if they have the same absolute invariants,
provided that the gcd of the weights of the non-zero coordinates among $ \{\xi_0(\phi),\dots,\xi_5(\phi)\} $ equals 1.
\end{prop}

\begin{proof}
Let $\phi,\psi\in\Rat_3^1$ be given by $\phi=\frac {f_0} {f_1}$, $\psi=\frac {g_0} {g_1}$, and let   
$\xi(\phi)$   %=[\xi_0(\phi),\ldots,\xi_5(\phi)]$ 
and  
$\xi(\psi)$   %=[\xi_0(\psi),\ldots,\xi_5(\psi)]$ 
be their weighted projective  invariants as in \cref{xi}, with absolute invariants as in \cref{absinv}.
Let $S\subseteq\{0,1,2,3,4,5\}$ be the common support of the non-zero coordinates    among $\{\xi_0(\phi),\dots,\xi_5(\phi)\}$ 
%(the equality of the $i_j$ forces the  zero pattern to agree). 
By assumption $\gcd(w_i:i\in S)=1$.

If $\psi=\phi^\sigma$ for some $\sigma\in\PGL_2(k)$, then by \cref{thm-1}
we have $\I_\psi = \I_\phi^\sigma$ and $\J_\psi=\J_\phi^\sigma$, and since the
$\xi_i$ are $\SL_2(k)$--invariants,
\[
\xi_i(\psi)=\xi_i(\I_\psi,\J_\psi)=\xi_i(\I_\phi^\sigma,\J_\phi^\sigma)
=\xi_i(\phi).
\]
If $\sigma\in\GL_2(k)$ with $\det\sigma=\lambda$, then $I_6(\phi^\sigma)=\lambda^3 I_6(\phi)$ and
$\xi_i(\phi^\sigma)=\lambda^{w_i/2}\,\xi_i(\phi),$
where $(w_0,\ldots,w_5)=(2,2,3,3,4,6)$.
Hence
\[
i_j(\psi)
=\frac{\xi_{j-1}(\psi)^{12/w_{j-1}}}{I_6(\psi)^2}
=\frac{(\lambda^{w_{j-1}/2}\xi_{j-1}(\phi))^{12/w_{j-1}}}
       {(\lambda^3 I_6(\phi))^2}
=i_j(\phi).
\]
Thus conjugate maps have the same absolute invariants.

Now suppose $i_j(\phi)=i_j(\psi)$ for all $j=1,\ldots,5$.   For each $j$ we have
$\frac{\xi_{j-1}(\phi)^{12/w_{j-1}}}{I_6(\phi)^2}   =\frac{\xi_{j-1}(\psi)^{12/w_{j-1}}}{I_6(\psi)^2}$, 
hence
\begin{equation}\label{tag1}
\left(\frac{\xi_{j-1}(\phi)}{\xi_{j-1}(\psi)}\right)^{12/w_{j-1}}  =\left(\frac{I_6(\phi)}{I_6(\psi)}\right)^2.
\end{equation}
Set  $\rho_j=\frac{\xi_{j-1}(\phi)}{\xi_{j-1}(\psi)}$ and $c=\left(\frac{I_6(\phi)}{I_6(\psi)}\right)^2$.   Since $k$ is algebraically closed with  $\chara (k)=0$, choose   $\lambda\in k^\times$ with $\lambda^{12}=c$.   Then every solution of \cref{tag1} has the form
\begin{equation}\label{tag2}
\rho_j = \lambda^{w_{j-1}}\zeta_j,     \qquad  \zeta_j^{\,12/w_{j-1}} = 1,
\end{equation}
because $12/w_{j-1}$ is integral for all weights $w_{j-1}\in\{2,2,3,3,4,6\}$.     The absolute invariants $i_1,\ldots,i_5$ are weight--zero rational functions
of the weighted coordinates $\xi_0,\ldots,\xi_5$.
Every monomial occurring in their numerators and denominators has weighted    degree
\begin{equation}\label{tag3}
\sum_{i=0}^5 a_i w_i = 0.
\end{equation}
Substituting the expressions \cref{tag2} into such a monomial (with support in $S$)   yields the factor $\prod_{i\in S}\zeta_i^{a_i}.$
Since $i_j(\phi)=i_j(\psi)$ for every $j$, all such monomials must satisfy
\begin{equation}\label{tag4}
\prod_{i\in S}\zeta_i^{a_i}=1
\end{equation}
for every exponent vector $(a_i)_{i\in S}$ satisfying the degree--zero
condition \cref{tag3}.

The vectors $(a_i)_{i\in S}$ arising from the weight--zero monomials span
the kernel of the restricted weight map
\[
\Z^S \longrightarrow \Z,
\qquad
(a_i)_{i\in S}\longmapsto \sum_{i\in S} a_i w_i.
\]
Because $\gcd(w_i:i\in S)=1$, this kernel contains no nontrivial torsion.
Therefore the only solution of \cref{tag4} is
$
\zeta_i=1$ for all $i\in S$.
Hence $\rho_j=\lambda^{w_{j-1}}$ for all $j$, i.e. $\xi_{j-1}(\phi)=\lambda^{\,w_{j-1}}\,\xi_{j-1}(\psi).$
Thus $\xi(\phi)=\lambda\star\xi(\psi)$, so $\phi$ and $\psi$ represent the same
point of the weighted projective space $\P(2,2,3,3,4,6)$ and therefore
the same conjugacy class in $\PGL_2(k)$.
\end{proof}

%------
We now define the map $\Phi: \C^8 \setminus \{ I_6 = 0 \} \longrightarrow \C^5$
by
\begin{equation}\label{Phi}
(c_0,\ldots,c_7) \to \left(i_1(c_0,\ldots,c_7),\ldots,i_5(c_0,\ldots,c_7)\right),
\end{equation}
where \(i_1,\ldots,i_5\) are the absolute invariants defined in \cref{absinv}.
The condition \(I_6\neq0\) ensures that these invariants are defined everywhere on the moduli space.

To study the image \(\Img(\Phi)\), we consider the Jacobian matrix \(J(\Phi)\), a
\(5\times8\) matrix with entries
\[
\frac{\partial i_j}{\partial c_k},
\qquad j=1,\ldots,5,\; k=0,\ldots,7.
\]
Since the rank of \(J(\Phi)\) is at most \(5\), it governs the local geometry of the map~\(\Phi\).
If \(\mathrm{rank}(J(\Phi))=5\), then \(\Phi\) is a submersion at the corresponding point, and by the Implicit Function Theorem the image \(\Img(\Phi)\) is locally a smooth $5$--dimensional manifold near \(\Phi(c_0,\ldots,c_7)\).

If \(\mathrm{rank}(J(\Phi))<5\), that is, if all \(5\times5\) minors vanish, then \(\Phi\) fails to be a submersion and the image may develop singularities.
Geometrically, these points correspond to rational functions possessing nontrivial automorphisms.
Equivalently, they are precisely the loci where the map described in \cref{prop-1} fails to be injective, i.e., where the gcd condition of  \cref{prop-1} is violated.

Thus the degeneracy of the Jacobian reflects the presence of additional symmetry.
This suggests that the singular locus of \(\Img(\Phi)\),
characterized by the condition \(\mathrm{rank}(J(\Phi))<5\),
coincides with the union of the automorphism strata.
In particular, we expect it to align with the locus \(\{J_6=0\}\subset\C^8\).
While a direct computational verification of this statement is difficult, the geometric picture is clear: the points where \cref{prop-1} fails correspond exactly to maps with extra automorphisms. These loci form the automorphism strata \(\L_1,\ldots,\L_7\), whose structure and explicit parametrizations are described in the next section.

 \section{Automorphisms} 
In this section we will define and study the automorphism groups of rational functions.  Such groups have long been the focus of many authors in arithmetic dynamics; see \cite{deg-3-4, miasnikov, faria, west}.   The approach here draws on the analogy with  the automorphism groups of hyperelliptic curves; see \cite{2002-3, 2006-4}.    First we recall some preliminaries. 

Let $G$ be a finite subgroup of $\PGL_2(k)$. Therefore $G$ is isomorphic to one of the following: $C_n$, $D_n$, $A_4$, $S_4$, or $A_5$. 
%Since each is embedded in $\PGL_2(k)$, 
We can represent their generators as matrices:
%, informing the forms of $\sigma \in \Aut(\phi)$ later. Below we display all the cases:
%
\begin{small}
\begin{equation}\label{eq1}
\begin{split}
& i) \;  C_n   \cong  \left< \begin{bmatrix}   \zeta_n & 0 \\ 0 & 1 \\ \end{bmatrix} \right>,    \; 
ii) \; D_n  \cong  \left< \begin{bmatrix} 0 & 1 \\ 1 & 0 \end{bmatrix},  \begin{bmatrix} \zeta_n & 0 \\ 0 & 1 \end{bmatrix} \right>, \; 
iii) \; A_4 \cong  \left< \begin{bmatrix}  -1 & 0 \\ 0 & 1 \\ \end{bmatrix},\begin{bmatrix} 1 & \zeta_4  \\ 1 & - \zeta_4  \\ \end{bmatrix} \right>  \\
&   iv) \; \quad S_4  \cong  \left<\begin{bmatrix}   -1 & 0 \\ 0 & 1 \\ \end{bmatrix},\begin{bmatrix}   0 & -1 \\ 1 & 0 \\ \end{bmatrix},\begin{bmatrix} -1 & -1 \\ 1 &  1 \\ \end{bmatrix}\right>, \; \;
v) \;  A_5  \cong \left<\begin{bmatrix}   \omega & 1 \\ 1 & -\omega \\ \end{bmatrix},\begin{bmatrix}   \omega & \zeta_5^4 \\ 1 & - \zeta_5^4 \omega \\ \end{bmatrix}\right>  
\end{split}
\end{equation}
\end{small}
where $\omega = \frac{-1 + \sqrt{5}}{2}$, $\zeta_n$ is a primitive $n^{th}$ root of unity
%, $\zeta_5$ is a primitive $5^{th}$ root of unity, 
and $i=\sqrt{-1}$.
% is a primitive $4^{th}$ root of unity.

\begin{rem}\label{fix-lem}
In each case above, there is $\sigma \in G$ which fixes $0$ and $\infty$. The proof is elementary. In the first two cases, the Möbius transformation $\sigma(x) = \zeta_m x$ will fix $0$ and $\infty$. In the next two cases, $\sigma(x) = -x$ will do that, and in the last case, $\sigma(x) = \omega x$ is in the group and will fix $0$ and $\infty$. This property constrains $\Aut(\phi)$ later.
\end{rem}

%*****************************************************
%\subsection{Automorphisms of rational maps}
%
Let $\phi \in \Rat_d^1$. A point $[x_0 : x_1] \in \P^1$ is called a \textbf{fixed point} for $\phi$ if $\phi(x_0, x_1) = (x_0, x_1)$. Let $t = x_1 / x_0$. Hence, $\phi$ can be taken as a rational function in $t$, say 
$\phi(t) = \frac{F(t)}{G(t)}$. 
Then $t$ is a fixed point if $\phi(t) = t$, which implies that 
\[
S(t) := F(t) - t \, G(t) = 0 
\]
which is at most a degree $(d+1)$ equation in $t$. Hence, a degree $d$ rational function has at most $(d+1)$ fixed points.

We denote the set of fixed points of $\phi$ by $\Fix(\phi)$. Notice that if $\Fix(\phi) = \{w_1, \ldots, w_s\}$ with $s \leq d+1$ is known, then we can uniquely determine the rational function $\phi$ by solving the linear system $F(w_i) - w_i \, G(w_i) = 0$, for $i = 1, \ldots, s$, assuming $s = d+1$. A function $\phi$ has less than $d+1$ fixed points exactly when the discriminant $\Delta(S, t) = 0$.

An \textbf{automorphism} of $\phi$ is called a $\sigma \in \PGL_2(k)$ such that $\phi \circ \sigma = \sigma \circ \phi$. 
%
\iffalse
\[
\xymatrix{
\P_x^1 \ar[d]_\phi    \ar[r]^\sigma    &    \P_x^1 \ar[d]^\phi  \\
\P_z^1  \ar[r]_\sigma   &   \P_z^1 \\
}
\]
%
\fi
%
The set of automorphisms of $\phi$ is denoted by
\[
\Aut(\phi) := \{ \sigma \in \PGL_2(k) : \phi^\sigma = \phi \}.
\]
It forms a group. For any $\sigma \in \Aut(\phi)$, by $\Fix(\sigma)$ we denote its set of fixed points.

\begin{rem}
As in the case of curves, there is some confusion in the literature on what is called the \textbf{automorphism group} of $\phi$. Throughout this paper, by $G := \Aut(\phi)$, we mean the \textbf{full automorphism group} of $\phi(x)$ over the algebraic closure of $k$ and not simply some $G \hookrightarrow \Aut(\phi)$ as it is used frequently by many authors.
\end{rem}

For an automorphism $\sigma \in \PGL_2(k)$ of order $s$, we denote by 
$\mathcal{L}_d^{\sigma}$ the locus of degree-$d$ rational maps fixed by $\sigma$,
\[
\mathcal{L}_d^{\sigma}
   = \{\, \varphi \in \mathrm{Rat}^1_d : \varphi^\sigma = \varphi \,\}.
\]
Now we have the following:

\begin{lem} \label{fixed-pts}
Let $\phi \in \Rat_d^1$ and $\sigma \in \Aut(\phi)$. Then
\begin{enumerate}[\upshape(i), nolistsep] 
\item If $p \in \Fix(\sigma)$, then $\phi(p) \in \Fix(\sigma)$.
\item If $w \in \Fix(\phi)$, then $\sigma(w) \in \Fix(\phi)$.
\end{enumerate} 
Hence, $\Aut(\phi)$ acts on $\Fix(\phi)$ by permutation. Moreover, if $\sigma \in \Aut(\phi)$ is an automorphism of order $m$, then $m$ divides the cardinality of $\Fix(\phi) \setminus \Fix(\sigma)$. Then
$\dim \mathcal{L}_d^{\sigma} = \delta = s - 1$, where $s$ is the number of orbits on $\Fix(\phi) \setminus \Fix(\sigma)$.
\end{lem}

\proof 
For any $p \in \Fix (\phi)$, $\s (p) \in \Fix (\phi)$ since 
\[  \phi (\s (p)) = \s (\phi (p) ) = \s (p), \]
which implies that $\s (p) \in \Fix (\phi )$.    If $w\in \Fix (\phi)$ then 
\[
\s (w) = \s ( \phi (w)) = \phi ( \s (w)), 
\]
which implies that $\s (w) \in \Fix (\phi)$.

Since $\< \s \>$ has no fixed points in $\Fix (\phi)\setminus \Fix (\s)$, then it acts transitively on $\Fix (\phi)\setminus \Fix (\s)$. Hence, $|\s|$ divides its cardinality. 

Since any automorphism $\sigma \in \PGL_2(k)$ of order $s$ is conjugate to one 
fixing $0$ and $\infty$, and the relation $\varphi^\sigma=\varphi$ is preserved 
under simultaneous conjugation of $\varphi$ and $\sigma$, we may assume 
without loss of generality that $\sigma$ fixes $0$ and $\infty$.
%We have fixed 0 and $\infty$ on $\P^1_x$. 
Hence,  the dimension of the family of rational functions $\phi(x)$  is one less than the number of roots of  $F$ and $G$.  Hence, this number is exactly $s-1$. 
\qed

\begin{prop}\label{phi-dec}
Let $\s\in  \Aut (\phi)$ such that $| \s |   =   m$. Then     $H:=\< \s\>$   acts on $\phi^{-1} (0)$ and $\phi^{-1} (\infty)$. Hence,   $\phi (x)$ can be written as 
\[
\phi(x) = x \, \psi (x^m)   \quad \text{ or }  %equivalently } 
\quad  \phi(x) = \frac 1 x \, \psi (x^m),  
\]
 where $\psi (x)$ is a rational function.   Moreover,  for $G\iso A_4, S_4, A_5$ then $m= 2, 4,   5$. 
\end{prop}

\proof    Let $\s \in G$ and  $t\in \Fix (\s)$.    For each $\a \in \phi^{-1} (t)$ we have 
\[
\phi (\s (\a)) = \s (\phi (\a )) = \s (t)=t.
\]
 Then $\< \s\>$ acts on the fiber $\phi^{-1} (t)$.    From \cref{fix-lem}   there is $\s \in G$ which fixes $0$ and $\infty$. Then $\<\s\>$ acts on $\phi^{-1} (0)$ and $\phi^{-1} (\infty)$.  Then points in $\phi^{-1} (0)$ and $\phi^{-1} (\infty)$      are 
\[
\begin{split}
&  \a_1, \xi \a_1, \ldots , \xi^{m-1} \a_1,  \;         \a_2, \xi \a_2, \ldots , \xi^{m-1} \a_2, \;     \ldots , \a_r, \xi \a_r, \ldots , \xi^{m-1} \a_r, \\
&  \b_1, \xi \b_1, \ldots , \xi^{m-1} \b_1,  \;         \b_2, \xi \b_2, \ldots , \xi^{m-1} \b_2, \;     \ldots , \b_r, \xi \b_r, \ldots , \xi^{m-1} \b_r, \\
\end{split}
\]
where $r=\frac {d-1} m$ and $\a_1, \ldots , \a_r, \b_1, \ldots , \b_r  \in k\setminus \{ 0,1,  \infty \}$. Therefore, 
\[
\begin{split}
\F (x) &=   \prod_{j=1}^r        \prod_{i=0}^{m-1} \left(   x - \xi_m^i \a_j \right)     =   \prod_{j=1}^r  (x^m -\a_j^m), \\
\G (x) &=    \prod_{j=1}^r        \prod_{i=0}^{m-1} \left(   x - \xi_m^i \b_j \right)     =    \prod_{j=1}^r  (x^m -\b_j^m) 
\end{split}
\]
Thus, $\phi (x)$  can be written as 
\[
\phi (x) =   x \,  \frac {\F(x^m)}  {\G(x^m)}    =
x \, \frac {  x^{rm} + a_{r-1} x^{(r-1)m}  + \cdots a_1 x^m + a_0}     {  x^{rm} + b_{r-1} x^{(r-1)m}  + \cdots b_1 x^m + b_0}  
\]
or $\phi (x) =    \frac {\G(x^m)}   {x \, \F(x^m)} $. 
This completes the proof.
 \qed

The above Lemmas give us a way to determine $\phi (x)$ once an automorphism $\sigma \in \Aut (\phi)$ is given.  
%Since the goal of this paper is to investigate computational and machine learning techniques for $d=3$ for the rest of the paper we focus solely on this case. 

%**********
%\newpage
\subsection{Automorphism Groups of Rational Cubics}

The automorphism groups for cubic rational functions and their parametric families were determined in \cite{deg-3-4} and in \cite{west}. Here we give a brief treatment to sort out some mistakes and furthermore determine explicitly  families in terms of invariants $\xi_0, \ldots , \xi_5$ and in each case give a formula for the rational cubic defined over its field of moduli. 

\begin{lem} \label{maxorder}
Let $\phi \in \Rat_d^1$ with $d = 3$. Then the following hold: 
\begin{enumerate}[\upshape(i), nolistsep] 
\item Elements of $\Aut(\phi)$ have orders at most 4.
\item $\Aut(\phi)$ is isomorphic to one of the following:  $\{e\}, C_2, C_3, C_4, V_4, D_4$, or $A_4$.
\end{enumerate} 
\end{lem}

\begin{proof}
Suppose that $\sigma \in \Aut(\phi)$ is of order $n$.  
There is no loss of generality in assuming 
\[
\sigma = 
\begin{bmatrix} 
1 & 0 \\ 
0 & \zeta_n 
\end{bmatrix},
\qquad 
\zeta_n \text{ a primitive $n$-th root of unity}.
\]
Then $\phi^\sigma = \phi$ if and only if $\phi$ commutes with~$\sigma$.
Since $\varphi = F/G$ has degree $3$, the monomials 
$x^{r_0}y^{3-r_0}$ in $F$ and $x^{r_1}y^{3-r_1}$ in $G$ must transform by the 
same character under $\sigma(x)=\zeta_n x$.  
This forces 
\[
\zeta_n^{\,r_0 - r_1 - 1} = 1,
\qquad\text{i.e. } n \mid (r_0 - r_1 - 1).
\]
Thus only those pairs of monomials with $n \mid (r_0 - r_1 - 1)$ can appear
simultaneously in $(F,G)$.

As a result, we only need to characterize cubic rational functions that have an automorphism of order 2, 3, or 4 in terms of their invariants (recall that if $\sigma \in \Aut(\phi)$ has order 4, then $\sigma^2 \in \Aut(\phi)$ has order 2). From (i), we eliminate $C_n$ and $D_{2n}$ for $n \geq 5$, as well as $A_5$, from consideration. Thus, possible groups are $\{1\}$, $C_2$, $C_3$, $C_4$, $V_4$, $D_4$, and $A_4$, as evidenced by the families below. We exclude $S_4$ as shown in the rest of this section. 
\end{proof}

Assume $\phi \in \Rat_3^1$ as in \cref{phi-deg3}.   Let $W := \Fix(\phi)$ be the set of fixed points, and $\sigma \in \Aut(\phi)$ is a non-trivial automorphism of order $n$. By \cref{maxorder} we have  $n = 2, 3,$ or 4.   The case of $n=4$ is a subcase of $n=2$.  Hence, there are two main cases $n=2$ and $n=3$.

%**********************************************
\subsection{Involutions}\label{sec:inv}
Let $\sigma \in \Aut (\phi)$ such that $|\sigma| = 2$. We can pick a coordinate in $\P^1$ as in \cref{eq1}  such that 
$
\sigma(z) = -z.
$
 Then $\Fix(\sigma) = \{0, \infty\}$. From \cref{fixed-pts}, we have that $\phi(0), \phi(\infty) \in \Fix(\sigma)$. Hence, there are two cases:
 $\phi$ fixes points of $\Fix(\sigma)$  or   $\phi$ permutes points of $\Fix(\sigma)$  and each of them corresponds to an irreducible surface as we will see next.

%********************************************
\subsubsection{$\phi$ fixes points of $\Fix(\sigma) = \{0, \infty\}$}
In this case, \(\phi(z)\) fixes both points in \(\Fix(\sigma)\), so \(\phi(0) = 0\) and \(\phi(\infty) = \infty\). By \cref{phi-dec}, \(\phi(z) = z \psi(z^2)\), where \(\psi(z^2)\) is a rational function ensuring \(\phi\) is degree 3 with \(\sigma(z) = -z\) as an automorphism. Hence, 
\[
\phi(z) = z \psi(z^2) = \frac{z (z^2 + a)}{z^2 + b},
\]
for \(a, b \in k\).  Assuming \(b \neq 0\) (to avoid a pole at \(z = 0\)), rewrite $\phi(z)$ as 
\[
\phi(z) = \frac{1}{b} \cdot \frac{z (z^2 + a)}{\frac{1}{b} z^2 + 1} = \frac{1}{b} \cdot \frac{z^3 + a z}{\frac{1}{b} z^2 + 1}.
\]
Define \(t = a\) and \(s = \frac{1}{b}\) and we have 
\[
\phi(z) = \frac{1}{b} \cdot \frac{z^3 + t z}{s z^2 + 1}.
\]
Since \(\phi = \frac{F(z)}{G(z)}\) is defined up to a scalar multiple, we can scale by \(b\), yielding:
\[
\phi(z) = z \frac{z^2 + t}{s z^2 + 1},
\]
for \(t, s \in k\), with \(s \neq 0\). This matches the family in \cite{deg-3-4}. 
Then   
\[
I_6(\phi) = t^2 s^2 - 2 t s + 1 = (t s - 1)^2 \neq 0.
\]
Hence,    in this case   $\phi$ is conjugate to a rational function written in the form 
\begin{equation}\label{phi:L1}
\boxed{\phi(z) = \frac{z^3 + t z}{s z^2 + 1}, \qquad \p = [1 : 0 : t : 0 : 0 : s : 0 : 1]}
\end{equation}
for some $t, s \in k$ such that $t s \neq 1$.    
%

%*******************************************************************************
\subsubsection{$\phi$ permutes points of $\Fix(\sigma)$}
We are still under assumption that there is an involution $\sigma \in \Aut (\phi)$ such that $\sigma (z)= - z$. Assume that $\phi$ permutes $0$ and $\infty$. 
 From \cref{phi-dec}, we have that $\phi(z) = \frac{1}{z} \psi(z^2)$, for some degree two rational function $\psi(z)$.   
  \cref{fixed-pts}  implies that   $\s$ permutes points in  $\phi^{-1} (0)$ and $\phi^{-1} (\infty)$ which are the numerator and denominator of $\psi$.  
 Thus 
\[
\phi(z) = \frac{1}{z} \, \frac{z^2 + b}{z^2 + a}.
\]
Similarly as the previous case, we can normalize $b=1$ and get  
\begin{equation}\label{phi:L2}
\boxed{\phi(z) = \frac{s z^2 + 1}{z^3 + t z}, \qquad \p = [0 : s : 0 : 1 : 1 : 0 : t : 0]}
\end{equation} 
for some $t, s \in k$. The resultant in this case is
\[
I_6(\phi) = - (t s - 1)^2 \neq 0
\]
which implies that $ts\neq 1$. 

\begin{rem}
Notice that $\phi(z)$ in \cref{phi:L2}   is the reciprocal of the function $\phi$ given in \cref{phi:L1}.

\end{rem}
%********************************************
\subsection{Extra involutions}
Assume that there is another involution $\tau \in \Aut(\phi)$.  Since $\s$ fixes 0 and $\infty$, then $\tau$ must permute. Hence $\tau (z) = \frac 1 z$. 
Thus we have the Klein 4-group $V_4$ embedded in the automorphism group  $\Aut(\phi)$.

\begin{prop}\label{prop:V4}
Let \(\phi(z)\) be a cubic rational function with an involution, written as
\[
\phi(z) = \frac{a z^2 + b}{z (c z^2 + d)},
\]
where \(a, b, c, d \in k\). If \(\phi(z)\) has another involution, then \(\phi(z)\) is conjugate under \(\PGL(2, k)\) to one of:
\[
\phi(z) = \frac{t z^2 + 1}{z^3 + t z} \quad \text{or} \quad \phi(z) = \frac{s z^2 - 1}{z^3 - s z},
\]
for some \(t, s \in k\).
\end{prop}

\begin{proof}
%Assume \(\phi(z) = \frac{a z^2 + b}{z (c z^2 + d)}\) is a degree 3 rational function and  \(\sigma(z) = -z\) and \(\tau(z) = \frac{1}{z}\) are involutions in \(\Aut(\phi)\).
%
Let us see what condition on coefficients $a, b, c, d$ are enforced by \(\tau(z) = \frac{1}{z}\).   
\[
\phi(\tau(z)) = \phi\left( \frac{1}{z} \right) = \frac{a \left( \frac{1}{z} \right)^2 + b}{\frac{1}{z} \left( c \left( \frac{1}{z} \right)^2 + d \right)} = \frac{\frac{a}{z^2} + b}{\frac{c}{z^2} + \frac{d}{z}} = \frac{\frac{a + b z^2}{z^2}}{\frac{c + d z^2}{z^3}} = \frac{z (a + b z^2)}{c + d z^2},
\]
\[
\tau(\phi(z)) = \frac{1}{\frac{a z^2 + b}{z (c z^2 + d)}} = \frac{z (c z^2 + d)}{a z^2 + b}.
\]
Then $\phi(\tau(z)) = \tau(\phi(z)) $ implies 
%\[\frac{z (a + b z^2)}{c + d z^2} = \frac{z (c z^2 + d)}{a z^2 + b},\]
$
\frac{a + b z^2}{c + d z^2} = \frac{c z^2 + d}{a z^2 + b},
$
which implies 
%\[a^2 z^2 + a b + a b z^4 + b^2 z^2 = c^2 z^2 + c d + c d z^4 + d^2 z^2,\]
\[
a b z^4 + (a^2 + b^2) z^2 + a b = c d z^4 + (c^2 + d^2) z^2 + c d.
\]
Hence we must have  \(a b = c d\) and   \(a^2 + b^2 = c^2 + d^2\).
We solve this system of equations  for \(b, d\) in terms of \(a, c\).  Recall that $c\neq 0$, otherwise $\deg \phi < 3$.  Then we have $(b, d) = (c, a)$ or $(b, d) = ( - c,  - a)$.  
If $(b, d) = (c, a)$ we have 
$
\phi(z) = \frac {az^2 +c} {cz^3 + az} = \frac{t z^2 + 1}{z^3 + t z},
$
for $t=a/c$.   
If $(b, d) = ( - c,  - a)$ we have 
$
\phi(z) = \frac {az^2 -c} {cz^3 - az} = \frac{t z^2 - 1}{z^3 - t z},
$
for $t=a/c$.   
Thus, \(\phi(z)\) is conjugate to \(\frac{t z^2 + 1}{z^3 + t z}\) or \(\frac{s z^2 - 1}{z^3 - s z}\).
\end{proof}

\subsubsection{First case} 
Let us denote the locus of all $\phi(z)$ is in the first case of the \cref{prop:V4} by $\L_4$.  
Notice that if $\phi \in \L_2$  this case    implies that $t=s$.   
Thus, 
\begin{equation}\label{phi:L4}
\boxed{\phi(z) = \frac{t z^2 + 1}{z^3 + t z}, \quad \p = [0, t, 0, 1, 1, 0, t, 0]}
\end{equation}
Similarly one can show that  if  $\phi \in \L_1$ then we still get $t=s$, which is a confirmation that this locus will be  the intersection $\L_1 \cap \L_2$.  We will denote it by $\L_4$. 

We denote by  $u:=t^2$.
Then $I_6$ invariant is 
\[
I_6(\phi) =- (t^2-1)=   - (u - 1)^2 \neq 0.
\]

%*********** L_5

\subsubsection{Second case}   
Let us denote the locus of all $\phi(z)$ is in the first case of the Proposition by $\L_4$.  
Notice that if $\phi \in \L_2$  this case    implies that $t=-s$. 
\begin{equation}\label{phi:L5}
\boxed{\phi(z) = \frac{s z^2 - 1}{z^3 - s z}, \quad \p = [0, s, 0, -1, 1, 0, -s, 0]}
\end{equation}
The resultant here is $I_6(\phi) = -s (s^2 - 1)^2 \neq 0$ for $s \neq 0, \pm 1$. 

%*******************************
\subsubsection{Alternating group: $\Aut(\phi) \cong A_4$}

Suppose a Klein four subgroup $V_4 \subset \Aut(\phi)$ extends to an
$A_4$--subgroup.  From \cref{eq1}, we may take 
$\sigma(z) = -z$ and $\tau(z) = \frac{z + i}{z - i}$, which generate a copy of
$A_4$ inside $\PGL_2(k)$.  
In this case the corresponding rational map is
\begin{equation}\label{phi:L6}
\boxed{\phi(z) = \frac{z^3 - 3}{-3 z^2}}, 
\qquad 
\mathbf{p} = [1 : 0 : 0 : -3 : 0 : -3 : 0 : 0].
\end{equation}

Since the $A_4$--action on $\P^1$ is unique up to conjugacy, and the
space of degree--$3$ relative invariants for this action is one--dimensional,
the resulting cubic map is unique up to conjugacy.  Hence the locus of maps
with automorphism group $A_4$ consists of a single point in $\cM_3^1$.

%******************************************************
\subsubsection{Dihedral group $D_4$}
Here, we can assume that   $\s$ is an automorphism of the form $\s(z) = \zeta_4 z$.   where $\zeta_4$ is a fixed primitive $4$th root of unity in $k$, is an automorphism. 

Since $\s^2$ in $\Aut(\phi)$ acts as $\s(z) = -z$, a situation we discussed earlier, we can begin with a rational function of the form in $\L_1$ or $\L_2$, which corresponds to  
\[
 \p=[1: 0 : t: 0:  0 : s: 0: 1      ]    \quad    \text{ or }          \p=[0: t: 0: 1:  1 : 0: s:0 ]
\]
Verifying that $\s\in\Aut(\phi)$, we can confirm that the first  case does not yield any possible rational cubic functions, % try the points (1:\pm1) and (1:2)
while in the second case, it must be the case that $t=s=0$. Hence 
\begin{equation}\label{phi:L7}
\boxed{
\phi=\frac{1}{z^3}, \qquad \p=[0,0,0,1,1, 0, 0, 0]
}
\end{equation}
%

%**************************************************************
\subsection{An automorphism of order 3}
Let $\sigma$ be of order 3 taken as in \cref{eq1}, so $\sigma(z) = \zeta_3 z$. Then, by \cref{phi-dec}, we have $\langle \sigma \rangle$ acts on the fiber $\phi^{-1}(0)$, which implies that the numerator of $\phi(z)$ is a polynomial $p(z) = z^3 - t$, for some $t \in k^\times$. We can pick $0$ and $\infty$ in $\phi^{-1}(\infty)$. Since $\phi^{-1}(\infty)$ is an orbit of $\langle \sigma \rangle$ and $\sigma$ fixes them, then one of them must have multiplicity 2. Thus, we can take 
\begin{equation}\label{phi:L3}
\boxed{\phi(z) = \frac{z^3 - t}{z}, \quad \p = [1 : 0 : 0 : -t : 0 : 0 : 1 : 0]}
\end{equation}
for some $t \in k$. The resultant is $I_6(\phi) = -t^3$, which gives the condition that $t \neq 0$.
%

%***********************************************************************
\begin{table}[h!]
\renewcommand{\arraystretch}{2}
\begin{center}
\begin{tabular}{|c|c|c|c|c|c|c|}
\hline  
 & $G$ & $\phi(z)$ & $\p \in \P^7$ & dim & Eq. $\L_i$ & $\xi (\phi)$ \\
\hline \hline 
$\L_1$ & $C_2$ & $\frac{z^3 + t z}{s z^2 + 1}$ & $[1, 0, t, 0, 0, s, 0, 1]$ & 2 & \eqref{eq:L1} & $\xi_2=\xi_3=0$ \\
$\L_2$ & $C_2$ & $\frac{s z^2 + 1}{z^3 + t z}$  & $[0, s, 0, 1, 1, 0, t, 0]$ & 2 & \eqref{eq:L2} & $\xi_5=0$ \\  
$\L_3$ & $C_3$ & $\frac{z^3 - t}{z}$  & $[1, 0, 0, -t, 0, 0, 1, 0]$ & 1 & \eqref{eq:L3} & \\
$\L_4$ & $V_4$ & $\frac{t z^2 + 1}{z^3 + t z}$  & $[0, t, 0, 1, 1, 0, t, 0]$ & 1 & \eqref{eq:L4} & $\xi_2=\xi_3=\xi_5=0$ \\
$\L_5$ & $V_4$ & $\frac{t z^2 - 1}{z^3 - t z}$  & $[0, t, 0, -1, 1, 0, -t, 0]$ & 1 & \eqref{eq:L5} & $\xi_0=\xi_3=\xi_4=\xi_5=0$ \\
$\L_6$ & $A_4$ & $\frac{z^3 - 3}{-3 z^2}$  & $[1, 0, 0, -3, 0, -3, 0, 0]$ & 0 & \eqref{eq:L6} & $[0:0:18:0:0:0]$ \\
$\L_7$ & $D_4$ & $\frac{1}{z^3}$  & $[0, 0, 0, 1, 1, 0, 0, 0]$ & 1 & \eqref{eq:L7} & $[0:-2:0:0:0:0]$ \\
\hline
\end{tabular}
\end{center}
\caption{Automorphism loci of degree 3 rational functions}
\label{tab:deg3}
\end{table}

This completes all the cases.    We summarize all cases  in  \cref{tab:deg3}. 
These results  agree with results in  \cite{deg-3-4}. 
Justifying the last two columns of the table will be the focus of the next section.

Inclusions among the loci are described in diagram \cref{diag}.  Notice that in each node of the diagram \cref{diag} we also put the invariants which vanish in that locus.  Justifying these inclusions , computing the invariants, and describing each locus in terms of these invariants will be done in the next section.

\begin{small}
\begin{figure}[h!]
\centering
\begin{tikzpicture}[scale=.4, thick,
boxY/.style={rectangle,draw=black!50,fill=yellow!20},
boxR/.style={draw=black!50,fill=red!60},
boxG/.style={draw=black!50,fill=green!60},
boxYY/.style={draw=black!50,fill=yellow!60}
]

% level labels
\node[boxY] at (-4,0)  {0};
\node[boxY] at (-4,4)  {1};
\node[boxY] at (-4,8)  {2};
\node[boxY] at (-4,14) {4};

% nodes
\node (node0) at (10,14) [boxR] {$\{e\}:\;\M_3^1$};

\node (node1) at (2,8)  [boxG] {\shortstack{$C_2:\;\L_1$\\ $\xi_2,\xi_3$}};
\node (node2) at (10,8) [boxG] {\shortstack{$C_2:\;\L_2$\\ $\xi_5$}};

\node (node3) at (20,4) [boxYY] {$C_3:\;\L_3$};

\node (node4) at (4,4)  [boxG] 
{\shortstack{$V_4:\;\L_4$\\ $\xi_2,\xi_3,\xi_5$}};

\node (node5) at (12,4) [boxG] 
{\shortstack{$V_4:\;\L_5$\\ $\xi_0,\xi_3,\xi_4,\xi_5$}};

\node (node6) at (16,0) [boxYY] 
{\shortstack{$A_4:\;\L_6$\\ $\xi_2\neq0$}};

\node (node7) at (8,0)  [boxG] 
{\shortstack{$D_4:\;\L_7$\\ $\xi_1\neq0$}};

% edges
\draw[-,blue!80] (node1.south) -- (node4.north);
\draw[->,blue!80] (node2.south) -- (node4.north);
\draw[-,blue!80] (node2.south) -- (node5.north);

\draw[-,blue!80] (node4.south) -- (node7.north);
\draw[-,blue!80] (node5.south) -- (node6.north);
\draw[-,blue!80] (node5.south) -- (node7.north);
\draw[-,blue!80] (node3.south) -- (node6.north);

\draw[-,blue!80] (node0.south) -- (node1.north);
\draw[-,blue!80] (node0.south) -- (node2.north);
\draw[-,blue!80] (node0.south) -- (node3.north);

\end{tikzpicture}
\caption{The inclusions among the loci for degree 3 rational functions}
\label{diag}
\end{figure}
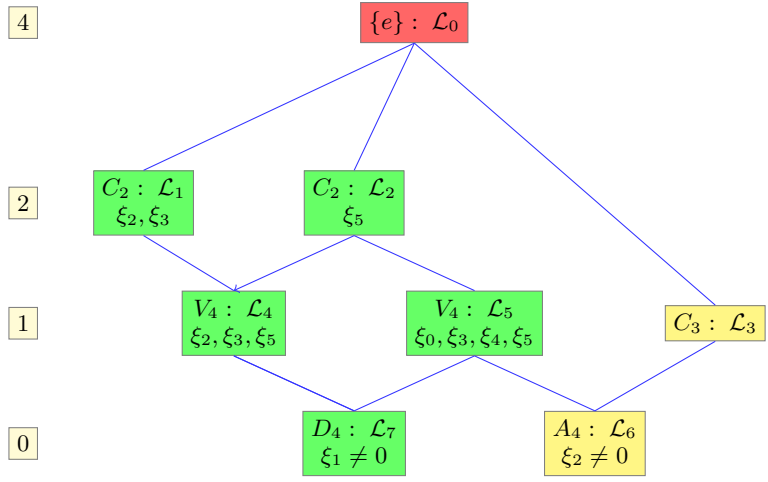
\end{small}

%**********
\begin{rem}
Those familiar with automorphisms groups of hyperelliptic curves notice the similarity between the two problems.  The zeroes and poles of $\phi (z)$ now play the role of the Weierstrass points of the hyperelliptic curve. While the reduced automorphism group of a hyperelliptic curve acts on the set of Weierstrass points, in our case here $\Aut (\phi)$ acts on the set of zeroes of $\phi (z)$ and on the set of poles. A general treatment of automorphism group of $\phi(z)$ can be done similarly to that of hyperelliptic curves; see \cite{2002-3}. 
\end{rem}

%*************
%*****************************************************
\section{Computation of the Loci with Fixed Automorphism Group}
In this section, we compute the loci of rational cubics with a fixed automorphism group in terms of invariants of rational cubics. We follow an analogy with algebraic curves and build on previous work by the third author in computing similar loci in the moduli space of genus 2 curves; see  
\cite{2000-1, 2002-3, 2003-3, 2006-4}. 
%and continuing with several other papers.
%
The first question is where this locus lies and what type of invariants we should use to compute it. Let \(\L\) denote the image in \(\P_{(2,2,3,3,4,6)}\) of the map
\[
\Phi: (c_0, \ldots, c_8) \to \left( \xi_0(c_0, \ldots, c_8), \ldots, \xi_5(c_0, \ldots, c_8) \right).
\]
This image is a weighted projective variety in \(\P_{(2,2,3,3,4,6)}\). Alternatively, one can use the Veronese morphism in \cref{veronese} and compute the locus of such rational functions as a projective variety, or even as an affine variety in terms of absolute invariants \(i_1, \ldots, i_5\).

Computing these loci in terms of \(i_1, \ldots, i_5\) is straightforward but highly inefficient. We could express \(i_1, \ldots, i_5\) in terms of the parameters described for each family in the previous section and then use a Gröbner basis approach to eliminate those parameters. However, the degrees of the resulting equations are very large, making this approach computationally challenging even for the relatively simple families considered here.
A more efficient method is to compute the loci as subvarieties of weighted projective spaces. This requires a modification of Buchberger’s algorithm to handle computations in a weighted projective space.

%**************************************
\subsection{Gröbner Bases for Weighted Homogeneous Ideals}

Consider a subvariety in \(\P(w_0, \ldots, w_n)\) defined parametrically by a set of polynomials over \(m\) parameters. The coordinates are given by:
\[
x_i = \xi_i, \quad i = 0, \ldots, n,
\]
where each \(\xi_i\) is a polynomial in \(k[t_1, \ldots, t_m]\), and \(t_1, \ldots, t_m\) are parameters with assigned degrees (typically positive integers). The image of this parameterization is a variety whose dimension depends on \(m\); for \(m\) independent parameters, the dimension is typically \(m - 1\). To define this variety explicitly, we seek polynomial relations among \(x_0, \ldots, x_n\) that hold for all parameter values, and these relations must be weighted homogeneous with respect to the weights \((w_0, \ldots, w_n)\). In other words, the relations must be satisfied not only by tuples \((\xi_0, \ldots, \xi_n)\) but also by all \((\lambda^{w_0} \xi_0, \ldots, \lambda^{w_n} \xi_n)\).

To compute these relations, we construct the ideal
\[
I = \langle x_0 - \xi_0, x_1 - \xi_1, \ldots, x_n - \xi_n \rangle
\]
in the polynomial ring \(k[t_1, \ldots, t_m, x_0, \ldots, x_n]\). The defining equations of the parameterized variety’s image in \(\P(w_0, \ldots, w_n)\) are the elements of the elimination ideal:
\[
I \cap k[x_0, \ldots, x_n],
\]
which consists of all polynomials in \(x_0, \ldots, x_n\) satisfied by the parameterization. By choosing a monomial order that eliminates the parameters \(t_1, \ldots, t_m\), the Gröbner basis of \(I\) includes polynomials free of \(t_j\), representing the desired relations.

To facilitate elimination while respecting the weighted structure, we assign degrees to both parameters and coordinates, say \(\deg(t_j) = d_j\) and \(\deg(x_i) = w_i\). Typically, we set \(d_j = 1\).

The weighted degree lexicographic order (\texttt{wdeglex}) compares monomials first by their total weighted degree, then lexicographically, with parameters ordered above coordinates (e.g., \(t_1 > \cdots > t_m > x_0 > \cdots > x_n\)). This ensures that the Gröbner basis computation prioritizes elimination of the parameters, yielding a set of generators for \(I \cap k[x_0, \ldots, x_n]\) that define the variety’s image in the coordinate ring of \(\P(w_0, \ldots, w_n)\).

The relations obtained from the elimination ideal are algebraically correct but may not be weighted homogeneous, as their terms can have differing weighted degrees. To make them valid defining equations in \(\P(w_0, \ldots, w_n)\), a homogenization step is required:
\begin{enumerate}
    \item \textbf{Degree Computation}: For a relation \(F = \sum c_\alpha x^\alpha\), compute the weighted degree of each term as \(\deg(x^\alpha) = \sum_{i=0}^n w_i \alpha_i\).
    \item \textbf{Target Degree}: Determine a common degree \(d\), typically the least common multiple (LCM) of the degrees of all terms in \(F\), i.e., \(d = \text{LCM}(\deg(x^\alpha))\).
    \item \textbf{Exponent Adjustment}: Transform each term \(c_\alpha x^\alpha\) into \(c_\alpha x^{\alpha k_\alpha}\), where \(k_\alpha = d / \deg(x^\alpha)\), ensuring \(\deg(x^{\alpha k_\alpha}) = k_\alpha \cdot \deg(x^\alpha) = d\). Adjust coefficients as needed to preserve the relation’s zero set.
\end{enumerate}
This process yields a weighted homogeneous polynomial that defines the same variety but adheres to the grading of \(\P(w_0, \ldots, w_n)\).

\subsubsection{General Algorithm}
For a weighted projective space \(\P(w_0, \ldots, w_n)\) with weights \(w = (w_0, \ldots, w_n)\) and a parameterization \(x_i = \xi_i(t_1, \ldots, t_m)\):
\begin{enumerate}
    \item Construct the ideal \(I = \langle x_0 - \xi_0, x_1 - \xi_1, \ldots, x_n - \xi_n \rangle\) in \(k[t_1, \ldots, t_m, x_0, \ldots, x_n]\).
    \item Compute the Gröbner basis \(B\) of \(I\) using a weighted degree lexicographic order with weights \((d_1, \ldots, d_m, w_0, \ldots, w_n)\) and order \(t_1 > \cdots > t_m > x_0 > \cdots > x_n\).
    \item Extract the set \(R = \{ f \in B \mid f \in k[x_0, \ldots, x_n] \}\), the generators of the elimination ideal.

\item For each polynomial \(f \in R\):
 
i) If \(f\) is single-term or all terms have the same weighted degree, retain \(f\) as-is.

ii)  Otherwise, compute \(d = \text{LCM}(\deg(x^\alpha))\) over all terms in \(f\), and adjust \(f\) to \(F = \sum c_\alpha x^{\alpha k_\alpha}\), where \(k_\alpha = d / \deg(x^\alpha)\), ensuring \(F\) is weighted homogeneous of degree \(d\).

    \item Define the ideal \(J = \langle R_{\text{homog}} \rangle\) in \(k[x_0, \ldots, x_n]\), which specifies the subvariety in \(\P(w_0, \ldots, w_n)\).
\end{enumerate}

This method uses Gröbner bases to eliminate parameters and enforces weighted homogeneity, providing a computational framework for studying subvarieties in weighted projective spaces with arbitrary weights \((w_0, \ldots, w_n)\). The method is described in \cite{faugere} and related works by its authors.

In our situation, however, there is an additional consideration. Our invariants \(\xi_0, \ldots, \xi_5\) are homogeneous polynomials when applied to the definition of \(\phi\) in \cref{phi}, meaning they are homogeneous in \(c_0, \ldots, c_7\). However, they are not necessarily homogeneous with respect to our specializations for \(t\) and \(s\) from the previous section. Thus, we first homogenize these defining equations by introducing another variable, say \(t_0\), and then apply the above procedure.

Next, we explicitly compute the loci as weighted projective varieties for each case, illustrating the approach described above.

%**************************************
\subsection{Locus \(\L_1\)}
We assume that \(\phi(z)\) is given as in \cref{phi:L1}. Computing the invariants \(\xi(\phi)\), we obtain:
\[
\begin{split}
\xi_0 &= 2 (t + 3) (s + 3), \quad
\xi_1 = \frac{1}{2} (t - 1) (s - 1), \quad
\xi_2 = \xi_3 = 0, \\
\xi_4 &= -\frac{1}{3} \left( t^2 s^2 + 2 t^2 s + 3 t^2 + 2 t s^2 - 8 t s - 6 t + 3 s^2 - 6 s + 9 \right), \\
\xi_5 &= -(t - s)^2 (t s + t + s - 3).
\end{split}
\]
There is an involution permuting \(t\) and \(s\). Let \(u\) and \(v\) denote invariants of that involution, defined as:
\[
\boxed{u = t + s \quad \text{and} \quad v = t s.}
\]
We can express all invariants in terms of \(u\) and \(v\) as follows:
\[
\begin{split}
\xi_0 &= 2 (v + 3u + 9), \quad \xi_1 = \frac{1}{2} (v - u + 1), \quad \xi_2 = \xi_3 = 0, \\
\xi_4 &= -\frac{1}{3} \left( v^2 + 3 u^2 + 2 u v - 14 v - 6 u + 9 \right), \\
\xi_5 &= (u^2 - 4 v) (3 - u - v).
\end{split}
\]
Also, \(I_6 = (t s - 1)^2 = (v - 1)^2 \neq 0\), so \(v \neq 1\). Moreover, \(J_6\) is given by:
\begin{equation}\label{J6:L1}
J_6 = 4 (s + 3) (t + 3) (s t + s + t - 3)^2 = 4 (3u + v + 9) (u + v - 3)^2.
\end{equation}
By computing a Gröbner basis for the weighted homogeneous system as described above, we obtain the following degree 24 weighted hypersurface:
\begin{small}
\begin{equation}\label{eq:L1}
\begin{split}
\L_1: \quad & \xi_0^6 \xi_1^6 - 54 \xi_0^4 \xi_1^4 \xi_4^2 - \frac{27}{4} \xi_0^4 \xi_1^3 \xi_5^2 - 27 \xi_0^3 \xi_1^4 \xi_5^2 - 108 \xi_0^3 \xi_1^3 \xi_4^3 + 729 \xi_0^2 \xi_1^2 \xi_4^4 + \frac{729}{4} \xi_0^2 \xi_1 \xi_4^2 \xi_5^2 \\
& + \frac{729}{64} \xi_0^2 \xi_5^4 + 729 \xi_0 \xi_1^2 \xi_4^2 \xi_5^2 + 2916 \xi_0 \xi_1 \xi_4^5 + \frac{243}{8} \xi_0 \xi_1 \xi_5^4 + \frac{729}{2} \xi_0 \xi_4^3 \xi_5^2 + \frac{729}{4} \xi_1^2 \xi_5^4 \\
& + 1458 \xi_1 \xi_4^3 \xi_5^2 + 2916 \xi_4^6 + \frac{729}{2} \xi_4 \xi_5^4 = 0
\end{split}
\end{equation}
\end{small}

%We can also compute \(\L_1\) in terms of the absolute invariants by eliminating \(u\) and \(v\). These computations are lengthy but express \(u\) and \(v\) as rational functions of the absolute invariants. Thus, the map in \cref{Phi} becomes:
%\[
%\begin{split}
%\Phi_1: k^2 \setminus \{ v = 1 \} &\to k^3 \\
%(u, v) &\to (i_1, i_2, i_5),
%\end{split}
%\]
%which is invertible when \(J_6 \neq 0\), i.e., \((u + v - 3) (v + 3u + 9) \neq 0\). The map \(\Phi_1\) provides a birational parametrization of \(\L_1\).

We can also compute $\L_1$ in terms of the absolute invariants by eliminating $u$ and $v$. These computations are lengthy but show that, on the open subset where $J_6 \neq 0$, both $u$ and $v$ can be expressed as rational functions of $(i_1,i_2, i_5)$. Thus the parametrization  induces a rational map
\[
\Phi_1 : k^2 \setminus \{v = 1\} \longrightarrow \L_1 \subset k^3,\qquad  (u,v) \longmapsto (i_1,i_2,i_5),
\]
whose image is exactly the hypersurface $\L_1$, and which is birational with rational inverse given by the functions $u = u(i_1,i_2,i_5)$ and $v = v(i_1,i_2,i_5)$ on $L_1 \cap \{J_6 \neq 0\}$.

\begin{rem}
The invariants \(u = t + s\) and \(v = t s\) mirror dihedral invariants in genus 2 curves \cite{2000-1, 2003-3}, symmetric under swapping \(t\) and \(s\). Here, they are invariant under parameter rescaling and conjugation adjusting \(t\) and \(s\), simplifying computations akin to symmetric polynomials in root permutations—a pattern likely extending to higher-degree rational functions.
\end{rem}

If \(\xi_0 = 0\), then \(t = -3\) or \(s = -3\), and the moduli point is:
$
\xi(\phi) = [0 : 8 : 0 : 0 : 0 : 0],
$
a singular point in \(\L_1\) with automorphism group isomorphic to \(D_4\) (cf. \cref{eq:L7}).

\begin{rem}
Note that the factor \(u^2 = 4v\) (or equivalently \(t = s\)) makes \(\xi_5 = 0\). This locus corresponds to \(\L_4\) (cf. \cref{eq:L4}), where \(\phi(z)\) has an extra involution. Since \(\L_4\) lies in the intersection of \(\L_1\) and the next locus, we will discuss it in detail later.
\end{rem}

%**************************************
\subsection{Locus \(\L_2\)}
Assume that \(\phi\) is given as in \cref{phi:L2}. We follow the same approach as above. Computing \(\xi(\phi)\), we get:
\[
\begin{split}
\xi_0 &= -2 (t + s)^2, \quad 
\xi_1 = \frac{1}{6} \left( (t - s)^2 - 12 \right), \quad   
\xi_2 = -\frac{1}{36} (s - t) \left( (t - s)^2 + 36 \right), \\
\xi_3 &= \frac{2}{3} (s - t) (t + s)^2, \quad
\xi_4 = -\frac{1}{9} (t + s)^2 \left( (s - t)^2 + 12 \right), \quad
\xi_5 = 0.
\end{split}
\]
We define invariants:
\[
u := (t + s)^2 \quad \text{and} \quad v := s - t,
\]
and express the invariants in terms of \(u\) and \(v\) as follows:
\[
\begin{split}
\xi_0 &= -2 u, \quad \xi_1 = \frac{1}{6} (v^2 - 12), \quad \xi_2 = -\frac{1}{36} v (v^2 + 36), \\
\xi_3 &= \frac{2}{3} u v, \quad \xi_4 = -\frac{1}{9} u (v^2 + 12), \quad \xi_5 = 0.
\end{split}
\]
Here, we have:
\[
J_6(\phi) = -16 (s + t)^4 = -16 u^2 \quad \text{and} \quad I_6 = \frac{v^2 - u + 4}{4}.
\]
Using a Gröbner basis to eliminate \(u\) and \(v\), we determine the equations for \(\L_2\) in \(\P(2, 2, 3, 3, 4, 5)\) as follows:
\begin{equation}\label{eq:L2}
\L_2: \quad 
\left\{
\begin{aligned}
& \xi_5 = 0, \\
& \xi_0^2 \xi_1 + 3 \xi_0 \xi_4 - 3 \xi_3^2 = 0, \\
& \xi_0^2 \xi_2 + \frac{1}{2} \xi_0 \xi_1 \xi_3 - 3 \xi_3 \xi_4 = 0, \\
& \xi_0 \xi_1 \xi_4 - \xi_0 \xi_2 \xi_3 - \frac{1}{2} \xi_1 \xi_3^2 + 3 \xi_4^2 = 0
\end{aligned}
\right.
\end{equation}

The absolute invariants are:
\[
\begin{split}
i_1 &= \frac{1024 u^6}{(-v^2 + u - 4)^2}, \quad
i_2 = \frac{(v^2 - 12)^6}{2916 (-v^2 + u - 4)^2}, \quad
i_3 = \frac{v^4 (v^2 + 36)^4}{104976 (-v^2 + u - 4)^2}, \\
i_4 &= \frac{256 v^4 u^4}{81 (-v^2 + u - 4)^2}, \quad
i_5 = -\frac{16 u^3 (v^2 + 12)^3}{729 (-v^2 + u - 4)^2}.
\end{split}
\]
Assuming \(J_6 \neq 0\) (i.e., \(u \neq 0\)), we eliminate \(u\) and \(v\), finding:
\begin{equation}\label{L2:uv}
u = -\frac{1}{2} \xi_0, \quad v = 3 \frac{\xi_3}{\xi_0}.
\end{equation}
Similarly to the previous case, the parameters $(u,v)$ give a rational
parametrization of the locus $\L_2$ of maps with automorphism group $C_3$.
Eliminating $u$ and $v$ from the expressions for the absolute invariants
$(i_1,\ldots,i_5)$ yields the defining equation of $\L_2$ as a hypersurface in
$k^3$ (or in $k^5$ when all absolute invariants are retained).  On the open
subset where $J_6 \neq 0$, both $u$ and $v$ can be recovered as rational
functions of the absolute invariants, so the parametrization induces a
birational map
\[
\Phi_2 : k^2 \dashrightarrow \L_2,
\qquad
(u,v) \longmapsto (i_1,i_2,i_5).
\]
Thus $\L_2$ is an irreducible surface (2-dimensional variety) birationally
parametrized by $(u,v)$, and the field of moduli of $\phi(z)$ is
$k(u,v)$.

Consider now the case when:
\[
J_6 = -16 (s + t)^4 = -16 u^2 = 0,
\]
implying \(t = -s\). The function becomes:
\[
\phi(z) = \frac{s z^2 + 1}{z^3 - s z},
\]
and then \(\xi_0 = \xi_3 = \xi_4 = \xi_5 = 0\), \(\xi_1 = \frac{v^2}{6} - 1\), and \(\xi_2 = -v\). This corresponds to \(\L_5\), as we will see later.

This concludes the loci \(\L_1\) and \(\L_2\) for rational functions with involutions as described in \cref{sec:inv}. The remaining loci are of dimension one or zero, making their computations simpler. We will address each in detail below.

%****************
\subsection{Locus $\L_4$}

Assume \(\phi\) is as in \cref{phi:L4}, with homogeneous parameterization in \(\P_{(2,2,3,3,4,6)}\):
\[
\xi(\phi) = \left[ -8 t^2, \, -2 u^2, \, 0, \, 0, \, -\frac{16 t^2 u^2}{3}, \, 0 \right].
\]
In the affine patch (\(u = 1\)):
\[
\xi(\phi) = \left[ -8 t^2, \, -2, \, 0, \, 0, \, -\frac{16 t^2}{3}, \, 0 \right].
\]
The Gröbner basis with weighted degree lexicographic order is:
\begin{equation}\label{eq:L4}
\L_4: \quad 
\left\{
\begin{aligned}
& \xi_0 \xi_1 + 3 \xi_4 = 0, \\
& \xi_2 =  \xi_3 =  \xi_5 = 0
\end{aligned}
\right.
\end{equation}
This defines a 1-dimensional variety in \(\P_{(2,2,3,3,4,6)}\).
% (codimension 4). Sage reports an affine dimension of 2, but in projective space, it’s \(2 - 1 = 1\), matching a 1-parameter family. Both parameterizations satisfy these equations.

%***********  L_5
\subsection{Locus $\L_5$}Assume $\phi$ as in \cref{phi:L5}.
Its invariants are 
\[
\xi(\phi) = \left[ 0, \, \frac{2 (s^2 + 3)}{3}, \, -\frac{2 s (s- 3) (s + 3)}{9}, \, 0, \, 0, \, 0 \right]
\]
and $I_6=(s- 1)^2 (s + 1)^2$.  Here we use the absolute invariants $i_2$ and $i_3$ and have the following system 
\[
\left\{
\begin{split}
&  i_2 \,s^{4}-2 i_2 \,s^{2}+ i_2 -\frac{4}{9} s^{4}-\frac{8}{3} s^{2}-4  =0  \\
&  i_3 \,s^{4}-2 i_3 \,s^{2}+i_3 -\frac{4}{81} s^{6}+\frac{8}{9} s^{4}-4 s^{2} =0 
\end{split}
\right.
\]
By taking the resultant of these two polynomials with respect to $t$ we get the affine version of this 1-dimensional variety
\[
16 i_2^{3}-72 i_2^{2} i_3 +81 i_2 \,i_3^{2}-96 i_2^{2}+216 i_2 i_3 -36 i_3^{2}+144 i_2 -96 i_3 -64 =0
\]
By replacing for $i_2$ and $i_3$ their definitions we get  a degree 18 weighted hypersurface   
\begin{equation}\label{eq:L5}
\begin{split}
\L_5:  \quad  &  	72 \xi_1^{6} \xi_2^{2} -16 \xi_1^{9} +96 \xi_1^{6} I_6 -81 \xi_1^{3} \xi_2^{4}-216 \xi_1^{3} \xi_2^{2} I_6 -144 \xi_1^{3} I_6^{2}  +36 \xi_2^{4} I_6 \\
&   +96 \xi_2^{2} I_6^{2}+64 I_6^{3} =0 \\
\end{split}
\end{equation}

%*************************   L_7
\subsection{Locus $\L_7$}
Assume $\phi$ as in \cref{phi:L7}.
Its invariants are 
\begin{equation}\label{eq:L7}
\xi(\phi)= \left[   0, -2, 0, 0, 0, 0    \right] \equiv [0,1,0,0,0,0]
\end{equation}

Here we write $\L(G)$ for the locus of degree-$3$ rational maps whose full
automorphism group contains a subgroup conjugate to $G \subset \PGL_2(k)$.
In particular, $\L(C_4)$ and $\L(D_4)$ denote the loci corresponding to the
cyclic and dihedral subgroups of order $4$.  We also let 
$\sigma_0(z) = -1/z$, the involution generating the order-$2$ subgroup of $S_4$
that is not contained in any copy of $D_4 \subset S_4$.

%Notably, the involution $\s' (z) = \frac 1 z$ acts as an additional automorphism for $\phi=\frac{y^3}{x^3}$, confirming that the loci $\L (C_4)$ and $\L (D_4)$ are identical. Additionally, $\s_0\notin\Aut(\phi)$, implying that the locus of $\L (S_4)$ is empty.  

 For the map $\varphi(z) = y^3/x^3 = z^{3}$ in normalized coordinates, the
involution $\sigma' (z) = 1/z$ acts as an additional automorphism, so the loci
$\L(C_4)$ and $\L(D_4)$ coincide.  On the other hand, the involution
$\sigma_0(z) = -1/z$, which would be required to extend $D_4$ to $S_4$, does
not preserve $\varphi(z)$, and therefore no cubic rational map admits $S_4$ as
a subgroup of its automorphism group.  Hence $\L(S_4)$ is empty.

%*******  L_6
\subsection{Locus $\L_6$}
Assume $\phi$ as in \cref{phi:L4}. 
The moduli point is 
\begin{equation}\label{eq:L6}
\L_6: \quad \xi(\phi) = [0 : 0 : 18 : 0 : 0 : 0] \equiv [0 : 0 : 1 : 0 : 0 : 0]
\end{equation}
Furthermore, it should be noticed that in this case $J_6(\phi)=0$.  

%*******   L_3
\subsection{Locus $\L_3$} 

Assume $\phi$ as in \cref{phi:L3}.  Next, we compute its  invariants
\[
\xi(\phi) = \left[ -2, \, \frac{1}{6}, \, \frac{27 t + 2}{72}, \, -\frac{27 t + 2}{3}, \, \frac{54 t - 1}{9}, \, -\frac{t (27 t - 16)}{4} \right]
\]
Notice that $I_6=t\neq 0$.  
The absolute invariants are
\[
i_1 =  \frac {64} {t^6}, \; i_2= \frac 1 {46656 t^6}, \; i_3= \frac {(27 t + 2)^4}{ 26873856 t^2}, \;
i_4 = \frac { (27 t + 2)^4}{81 t^2}, \;
i_5= \frac {(54 t - 1)^3} { 729 t^2 }
\]
 By eliminating $t$ from these equations, we can express $t$ as a rational function in terms of $i_1$, $i_5$ and get the following affine curve
 %
 %     We should divide this polynomial by its content. It seems as all coefficients are divisible by at least 2
 \begin{small}
\begin{equation}
\begin{split}
& 614787626176508399616  i_1  i_5^6 + 44264709084708604772352  i_1  i_5^5 + 49589822592  i_1^2  i_5^3 \\
&  + 1150882436202423724081152  i_1  i_5^4  - 6248317646592  i_1^2  i_5^2  - 4760622968832  i_1^2 \\
& +  i_1^3 + 35704672266240  i_1^2  i_5 + 59491769009848364814041088  i_1  i_5^2   \\
& + 79322358679797819752054784  i_1  i_5 + 7554510350456935214481408  i_1 \\
& - 3996019499184929743169818581358608384 + 12984314664847857399889920  i_1  i_5^3 = \, 0 
\end{split}
\end{equation}
\end{small}
We can express this as a weighted  projective curve by substituting for $i_1$ and $i_5$:
 \begin{small}
\begin{equation}\label{eq:L3}
\begin{split}
 &	I_6^4 \xi_0^9 - 2834352 I_6^3 \xi_0^6 \xi_4^3 + 24794911296 \xi_0^3 \xi_4^9 + 3779136 I_6^5 \xi_0^6 + 892616806656 I_6^2 \xi_0^3 \xi_4^6 \\
 & + 7140934453248 I_6^4 \xi_0^3 \xi_4^3 + 4760622968832 I_6^6 \xi_0^3 + 1999004627104432128 I_6^7 =0
\end{split}
\end{equation}
\end{small}
This completes all the cases.

%\begin{rem}
%We computed each locus as a weighted projective variety in the weighted projective space $\P_{(2,2,3,3,4,6)} $. From the arithmetic point of view we are interested on rational points on these weighted projective varieties.  Rational points on weighted varieties  are discussed in  \cite{2023-01}  based on   weighted heights which give in general a more efficient approach then projective heights. 
%\end{rem}

%*******************
\section{A Database of Cubic Rational Functions}  \label{sec:ml-aut}

In light of the recent success of machine-learning techniques in algebraic geometry — particularly the neuro-symbolic classification of genus-two curves and their isogeny graphs \cite{2024-03}, the database-driven computation of Galois groups of polynomials \cite{2024-05}, and the earlier systematic enumeration of rational points on moduli spaces \cite{2016-5} — we apply the same methodology to the moduli space $  \M_3^1  $ of degree-three rational functions. Determining automorphism groups (and more generally the stratification of $  \M_d^1  $) becomes increasingly subtle as the degree $  d  $ grows, and the raw coefficient space $  \Rat_3^1 \subset \P^7  $ suffers from massive redundancy under $  \PGL_2  $-action. Following the approach that proved highly effective for hyperelliptic curves, we therefore construct a large, clean database of rational cubics over $  \Q  $ of naive height at most 4, represented directly by their weighted projective invariants $  \xi_0,\dots,\xi_5  $ in $  \P_{(2,2,3,3,4,6)}^5(\Q)  $. This database eliminates all conjugacy redundancies, provides a natural low-dimensional feature set, and serves as a testbed for automated classification of automorphism groups (and later for minimal fields of definition and inclusion relations among strata).

In this section, we construct a comprehensive database of cubic rational functions over the rational numbers \(\Q\), denoted \(\Rat_3^1\), leveraging the weighted projective space \(\P_\w^5(\Q)\) with weights \(\w = (2, 2, 3, 3, 4, 6)\) as a parametrization framework. This database, denoted \(\cP_3^h\), catalogs rational functions \(\phi(x) = \frac{f_0(x)}{f_1(x)} \in \Q(x)\) of degree 3, where \(f_0(x)\) and \(f_1(x)\) are polynomials of degrees 3 and 2, respectively, ensuring the degree of the rational function is \(\deg(\phi) = \deg(f_0) - \deg(f_1) = 3\). Each function is represented as a projective point in \(\P_\Q^7\), and we impose constraints on height and coprimality to define the dataset systematically.

A cubic rational function \(\phi(x) \) is given as in \cref{phi-deg3} and  corresponds to a point \(P_\phi = [c_0 : c_1 : \cdots : c_7] \in \P_\Q^7\). We define the naive height of \(\phi\) as 
\[
H(\phi) = \max \{ |c_i| \mid i = 0, \ldots, 7 \},
\]
 and restrict our dataset to functions with \(H(\phi) \leq h\), where \(h\) is a specified height bound. To ensure well-definedness, we require that the coefficients are coprime, i.e., \(\gcd(c_0, c_1, \ldots, c_7) = 1\), and that the resultant \(I_6(\phi) = \Res(f_0, f_1) \neq 0\), guaranteeing that \(\phi\) has no common roots between numerator and denominator. Thus, we define:
\[
\cP_3^h := \{ P_\phi \in \P_\Q^7 \mid H_\Q(P_\phi) \leq h, \, \gcd(c_0, \ldots, c_7) = 1, \, I_6(\phi) \neq 0 \}.
\]
For each \(\phi \in \cP_3^h\), we compute several invariants and properties to enrich the database. These include the automorphism group \(\Aut(\phi)\), determined by checking against the classification of possible groups for cubic rational functions (e.g., \(\{e\}\), \(C_2\), \(D_4\), etc.) as outlined in \cref{tab:deg3}, the invariants \(\mathbf{p} = [\xi_0, \xi_1, \xi_2, \xi_3, \xi_4, \xi_5]\) in \(\P_\w^5(\Q)\), and the absolute invariants \((i_1, i_2, i_3, i_4, i_5)\). Additionally, we calculate the invariant \(J_6(\phi)\) and the weighted moduli height \(\hat{h}(\phi)\), which measures the height of the invariants in the weighted projective space; as discussed in detail in \cite{2019-1, 2022-1}.  
In our case, for a given $\phi$ we compute \(\mathbf{p} = [\xi_0(f), \xi_1(f), \xi_2(f), \xi_3(f), \xi_4(f), \xi_5(f)]\ \in \P_\w^5(\Q)\)
 defined over a number field $K$ and    its \textbf{weighted moduli height} is
\[
\hat{h}(\phi)
   = \prod_{v} 
     \max_{0\le i \le 5} 
     \bigl\{ |\xi_i|_{v}^{1/w_i} \bigr\},
\qquad
(w_0,w_1,w_2,w_3,w_4,w_5)=(2,2,3,3,4,6),
\]
where the product is over all places $v$ of $K$.

The data is stored in a Python dictionary, with keys given by the coefficient tuples \((c_0, c_1, \ldots, c_7)\) and values as lists containing:

\begin{itemize}
    \item \(H(\phi)\): the naive height,
    \item \(\mathbf{p} = (\xi_0, \xi_1, \xi_2, \xi_3, \xi_4, \xi_5)\): the projective invariants,
    \item \(\hat{h}(\phi)\): the weighted moduli height,
    \item \(J_6(\phi)\): an additional invariant,
    \item \(\Aut(\phi)\): the automorphism group (e.g., '\{e\}'),
    \item \((i_1, i_2, i_3, i_4, i_5)\): the absolute invariants.
\end{itemize}

An example entry from the database is:
\[
\begin{split}
& (2, 3, -1, -3, 1, 2, -3, 1) \mapsto \left[ 3, (32, 12, 13, -164, -424, 2572), 5.66, 89360, '\{e\}',   \right.  \\
& \left. \left( \frac{1073741824}{44521}, \frac{2985984}{44521}, \frac{531441}{712336}, \frac{723394816}{44521}, -\frac{76225024}{44521} \right) \right].
\end{split}
\]

Here, the key \((2, 3, -1, -3, 1, 2, -3, 1)\) represents the coefficients of 
\[
\phi(x) = \frac{2 + 3x - x^2 - 3x^3}{1 + 2x - 3x^2 + x^3},
\]
 with \(H(\phi) = 3\), invariants \(\xi_i\) mapping to a point in \(\P_\w^5(\Q)\), a weighted height of 5.66, \(J_6 = 89360\), the trivial automorphism group \(\{e\}\), and the corresponding absolute invariants.

To illustrate the scope of the dataset,  \cref{tab:aut_distribution} shows   the distribution of automorphism groups for rational cubics of weighted moduli  height \(H(\phi) \le 3\).  The entries are grouped according to the loci  \(\M_3^1, \L_1, \ldots, \L_7\), corresponding respectively to the trivial group \(\{e\}\) and the nontrivial automorphism groups \(C_2\), \(C_3\), \(D_4\),  and so on, as classified earlier in the paper.

\begin{table}[h]
    \centering
    \caption{Distribution of   groups for cubics of height \(H(\phi) \leq 3\)}
    \begin{tabular}{|c|c|c|c|c|c|c|c|c|c|}
        \hline
        \(H\) & \(\M_3^1\)   & \(\L_1\)   & \(\L_2\)   & \(\L_3\)   & \(\L_4\)   & \(\L_5\) & \(\L_6\) & \(\L_7\)   & Total \\
        \hline
        1 & 2223 & 9 & 8 & 6 & 0 & 0 & 0 & 2 & 2248 \\
        2 & 84267 & 34 & 12 & 17 & 0 & 0 & 0 & 2 & 84332 \\
        3 & 814126 & 81 & 66 & 44 & 1 & 22 & 18 & 50 & 814408 \\
        \hline
        Total & 900616 & 124 & 86 & 67 & 1 & 22 & 18 & 54 & 900988 \\
        \hline
    \end{tabular}
    \label{tab:aut_distribution}
\end{table}
 
The distribution is strongly skewed toward the trivial automorphism group \(\{e\}\) (the locus \(\M_3^1\)), which accounts for 900{,}616 of the 900{,}988
distinct conjugacy classes of height \(H(\phi) \le 3\).  Nontrivial groups such as \(D_4\) (\(\L_3\)) and \(A_4\) (\(\L_7\)) occur much less frequently, in
accordance with their rarity among cubic rational maps.

The case \(H(\phi)=4\) is substantially more demanding computationally.   Our current search has identified at least \(350{,}679\) distinct moduli points in \(\cM_3^1\) of weighted moduli height \(4\), but the  enumeration is not yet certified to be complete.  We therefore record  \(350{,}679\) as a rigorous lower bound for the number of conjugacy classes  of height \(4\).  The fully completed data for heights \(0 \le h \le 3\)  appear in \cref{tab:aut_distribution}.

%The dataset contains $N_{\mathrm{train}} = 2,078,697$ samples in the training set and $N_{\mathrm{test}} = \dots$ samples in the testing set; the counts for  each stratum are listed in Table~\ref{tab:class_distribution}.

This database provides a robust foundation for subsequent analysis, including the machine learning classification of automorphism groups discussed in Section 8. The inclusion of both projective and absolute invariants, alongside height metrics and group labels, enables a detailed exploration of the geometric and arithmetic properties of \(\Rat_3^1\).

%*****************************************************************************************************
%\section{Machine Learning and Weighted Projective Clustering}\label{sec:ML}  
\section{Classification of  Groups   Using Machine Learning}

The machine–learning experiments in this section are intended solely as an  illustrative exploration of how the absolute invariants may be used within a 
computational classification framework.   No mathematical result in the paper depends on these experiments, and their  purpose is not to establish statistical performance guarantees but to  demonstrate that the invariant–theoretic structure developed in the earlier  sections lends itself naturally to algorithmic methods.  
The extreme class imbalance visible in small-height data is intrinsic to the  geometry of the moduli space and not a feature of the classifier.   Accordingly, the numerical results below should be interpreted as qualitative  evidence rather than as a benchmark evaluation of machine–learning models.

Our experiments aim to classify the automorphism group of a rational cubic  function by means of supervised learning.   Two approaches were tested: a baseline model using the coefficients of the  rational function as input features, and a refined model using the weighted   invariants.  
%The latter approach is more closely aligned with the moduli-theoretic structure,  since the invariants determine the conjugacy class uniquely;  see \cref{prop-1}. 

The dataset used here consists of rational functions of weighted moduli height  $H(\phi)\le 4$, partitioned by automorphism group.    Its distribution is extremely imbalanced: the generic case $\{e\}$ accounts for  over $99\%$ of all samples, reflecting the fact that rational cubics with  nontrivial automorphism group form very sparse loci in $\cM_3^1$.  
 \cref{tab:class_distribution} lists the proportions and the number of   samples appearing in the testing set for each group.  
\begin{table}[h]
    \centering
    \begin{tabular}{ccc}
        \toprule
        Automorphism Group & Proportion of Data & Samples in Test Set \\
        \midrule
        \{e\} (6)      & 99.84\%   & 179{,}910 \\
        C2-2 (2)       & 0.12\%    & 214 \\
        C2-1 (1)       & 0.02\%    & 29 \\
        D4 (3)         & 0.01\%    & 16 \\
        V4-1 (4)       & 0.0075\%  & 17 \\
        A4 (0)         & 0.0067\%  & 12 \\
        V4-2 (5)       & 0.0002\%  & 2 \\
        \bottomrule
    \end{tabular}
    \caption{Class distribution in the dataset.}
    \label{tab:class_distribution}
\end{table}

%----------------------------------------------------------------------------
\subsection{Initial Model: Using Coefficients as Features}
We formulate the classification of automorphism groups as a supervised learning 
problem, where the input features are the coefficients of the rational 
functions and the target variable is the corresponding automorphism group.  
A Random Forest classifier with 100 estimators was trained on these features.

The reported overall accuracy of $99.97\%$ is not meaningful in this setting, 
because the dataset is dominated by the generic class $\{e\}$, which accounts 
for $99.84\%$ of all samples.  A classifier that always predicts $\{e\}$ would 
already achieve almost the same accuracy.  
More informative performance metrics appear in 
Table~\ref{tab:initial_performance}.  
These show that, although the majority class is classified perfectly 
(precision, recall, and F1-score all equal to $1.00$), the minority classes 
exhibit very low recall—for instance, $0.14$ for $C2$--$1$ and $0.41$ for 
$V4$--$1$—indicating that the model frequently collapses the rare strata into 
the dominant class.  This behavior reflects the fact that raw coefficients do 
not encode the moduli-theoretic structure of $\cM_3^1$ and are 
therefore a poor feature set for distinguishing automorphism groups.
\begin{table}[h]
    \centering
    \begin{tabular}{cccc}
        \toprule
        Class & Precision & Recall & F1-score \\
        \midrule
        A4 (0)      & 1.00 & 0.58 & 0.74 \\
        C2-1 (1)    & 1.00 & 0.14 & 0.24 \\
        C2-2 (2)    & 0.97 & 0.91 & 0.94 \\
        D4 (3)      & 1.00 & 0.94 & 0.97 \\
        V4-1 (4)    & 0.78 & 0.41 & 0.54 \\
        \{e\} (6)   & 1.00 & 1.00 & 1.00 \\
        \bottomrule
    \end{tabular}
    \caption{Performance metrics   using coefficients as 
    features. 
    %The high overall accuracy reflects the dominance of the generic    class $\{e\}$.
    }
    \label{tab:initial_performance}
\end{table}
\subsubsection{Addressing Class Imbalance}
To partially mitigate the class imbalance, we apply class weighting to the 
Random Forest classifier.  
For each class~$i$ we assign a weight 
$w_i = \frac{N}{C \times n_i}$, where $N$ is the total number of samples, 
$C$ the number of classes, and $n_i$ the number of samples in class~$i$.  
This increases the influence of underrepresented classes during training.

After retraining with these weights, the reported accuracy remains 
$99.96\%$, but the performance across classes becomes more informative.  
As shown in Table~\ref{tab:weighted_performance}, recall improves for some 
minority groups—for example, $A4$ increases from $0.58$ to $0.83$—while other 
rare classes remain difficult to identify (e.g.\ $C2$--$1$ decreases to $0.10$ 
and $V4$--$1$ to $0.29$).  
This indicates that class weighting reduces, but does not eliminate, the 
limitations inherent in using coefficients rather than invariant-based 
representations.
\begin{table}[h]
    \centering
    \begin{tabular}{cccc}
        \toprule
        Class & Precision & Recall & F1-score \\
        \midrule
        A4 (0)      & 1.00 & 0.83 & 0.91 \\
        C2-1 (1)    & 1.00 & 0.10 & 0.19 \\
        C2-2 (2)    & 0.98 & 0.90 & 0.94 \\
        D4 (3)      & 1.00 & 0.94 & 0.97 \\
        V4-1 (4)    & 0.71 & 0.29 & 0.42 \\
        \{e\} (6)   & 1.00 & 1.00 & 1.00 \\
        \bottomrule
    \end{tabular}
    \caption{Performance metrics after applying class weighting to the 
    coefficient-based model.}
    \label{tab:weighted_performance}
\end{table}
%
%-----------------
\subsection{Using Invariants as Input Features}
Given the limitations of the coefficient-based approach—particularly its 
difficulty distinguishing minority strata even after class weighting—we train a 
second model using the weighted invariants introduced in Sections~4–6 as input 
features, following the approach of graded neural networks in  \cite{2024-2,   sh-89}. 
These invariants provide a complete moduli-theoretic description: by Lemma~3, 
they determine the conjugacy class of a rational cubic uniquely.  
Thus, unlike raw coefficients, they encode exactly the structure relevant for 
identifying the automorphism group.

A Random Forest classifier with 100 estimators was trained using the same 
training–testing split and the same class-weighting scheme as in the previous 
subsection.  
The results, summarized in Table~\ref{tab:invariants_performance}, show a 
substantial improvement across all minority classes.  
The reported accuracy 
\[
\text{Reported Accuracy} = 0.9999223076837701 \approx 99.992\%
\]
should again be interpreted in light of the strong class imbalance; however, in 
contrast to the coefficient-based model, the classwise precision and recall here 
provide meaningful information because the invariants separate strata in 
$\cM_3^1$.
\begin{table}[h]
    \centering
    \begin{tabular}{lcccc}
        \toprule
        Class & Precision & Recall & F1-score & Support \\
        \midrule
        A4 (0)      & 1.00 & 1.00 & 1.00 & 12 \\
        C2-1 (1)    & 1.00 & 0.93 & 0.96 & 29 \\
        C2-2 (2)    & 1.00 & 0.95 & 0.98 & 214 \\
        D4 (3)      & 1.00 & 1.00 & 1.00 & 16 \\
        V4-1 (4)    & 0.89 & 1.00 & 0.94 & 17 \\
        \{e\} (6)   & 1.00 & 1.00 & 1.00 & 179{,}910 \\
        \midrule
        \textbf{Accuracy}      &&& 1.00 & 180{,}198 \\
        \textbf{Macro avg}     & 0.98 & 0.98 & 0.98 & 180{,}198 \\
        \textbf{Weighted avg}  & 1.00 & 1.00 & 1.00 & 180{,}198 \\
        \bottomrule
    \end{tabular}
    \caption{Performance metrics using invariants as input features with 
    class weighting.  
%    Perfect classification of $A_4$ and $D_4$ reflects the fact that each of     these strata consists of a single conjugacy class.
    }
    \label{tab:invariants_performance}
\end{table}
The recall for minority classes improves dramatically: $C2$--$1$ increases from 
$0.10$ to $0.93$, $C2$--$2$ from $0.90$ to $0.95$, and the strata $A_4$, $D_4$, 
and $V4$--$1$ achieve perfect recall.  
The slight precision drop for $V4$--$1$ (to $0.89$) arises from a small number 
of false positives and is consistent with the small sample size.

It is important to note that perfect performance on $A_4$ and $D_4$ is not a 
reflection of model sophistication: each of these loci consists of a single 
conjugacy class in $\cM_3^1$, uniquely determined by its invariants.  
Similarly, high performance on $C2$--$2$ does not arise from a single criterion 
such as the vanishing of $\xi_5$, since the absolute invariants depend on all 
weighted coordinates (\cref{absinv}) and uniquely determine the moduli point.

In conclusion, the invariant-based model exhibits far better stability and 
classwise performance than the coefficient-based model.  
This improvement is a consequence of the moduli-theoretic information encoded by 
the invariants rather than of any machine-learning effect.  
These experiments therefore serve only to illustrate that the invariant theory 
developed earlier naturally lends itself to computational classification; the 
mathematical results of the paper do not depend on the performance of the 
models.

%********************
\section{Conclusions and Further Directions} \label{sec:conclusion}
A central motivation for this work is the structural analogy between the
moduli space of rational functions and the classical moduli spaces of
hyperelliptic and superelliptic curves.  
In both settings, the moduli problem is governed by the action of $\PGL_2$ on
binary forms (or on suitable tuples of forms), the resulting quotients admit
coordinate descriptions in terms of weighted projective invariants, and the
loci where the objects acquire extra automorphisms are cut out by explicit
relations among these invariants.  
For hyperelliptic curves this viewpoint underlies the rich invariant theory of
binary forms; in the degree-three dynamical setting, an analogous picture
arises for the pair $(F,G)\in V_4\oplus V_2$, whose invariants
$\xi_0,\dots,\xi_5$ provide weighted-projective coordinates on $\M_3^1$ and,
via the absolute invariants, determine the conjugacy class.

Within this framework, we identify $\M_3^1$ with the weighted projective space
$\P_\w^5$ (weights $(2,2,3,3,4,6)$), give explicit generators
$\xi_0,\dots,\xi_5$ for the invariant map, and characterize the loci
$\L_1,\dots,\L_7$ corresponding to the finite subgroups
$C_2, C_3, C_4, V_4, D_4,$ and $A_4$ of $\PGL_2$.  
The absolute invariants $i_1,\dots,i_5$ are shown to classify conjugacy
classes uniquely, providing a complete and computable coordinate system for
$\M_3^1$ analogous to the absolute Igusa invariants for genus-two curves.

The construction of the dataset $\cP_3^4$ further illustrates the
arithmetic distribution of these loci inside $\M_3^1$, with the trivial
automorphism locus dominating for small heights, just as in the hyperelliptic
case where curves with extra automorphisms form thin subsets.  
This reinforces the geometric picture: enhanced symmetry is exceptional and
detectable through the vanishing patterns of weighted invariants.

Section~8 provides only a brief computational illustration.  
No mathematical result of the paper depends on these experiments; they merely
show that the invariant map offers a natural feature representation for
algorithmic classification and that the stratification of $\M_3^1$ is
reflective in the numerical data.  
Their role is supplementary rather than foundational.

Several avenues for further research arise naturally.  
The analogy with hyperelliptic moduli suggests studying higher-degree rational
maps via invariant theory for $(V_{d+1}, V_{d-1})$, where the structure of the
ring of invariants remains largely unexplored.  
Extending the explicit description of automorphism loci to larger $d$,
constructing datasets of higher height, and developing symbolic or
neuro-symbolic methods for approximating invariants in degrees where
$\cR_{(d+1,d-1)}$ is not yet known are promising directions.

In summary, this paper establishes a complete invariant–theoretic description
of the moduli space $\M_3^1$, identifies all automorphism loci explicitly, and
demonstrates how the geometry of these loci parallels classical phenomena from
the theory of hyperelliptic and superelliptic curves.  
This provides both a conceptual bridge between two active areas of arithmetic
geometry and a foundation for further investigations in arithmetic dynamics
and invariant theory.

%% file: appendix.tex
\appendix
\section{Invariants of rational cubics}\label{app-a}
% We have the following expressions for invariants: 
%
\begin{small}
\[
\begin{split}
\xi_0 = &  2 (3 \c_2 \c_0 + \c_2 \c_5 - \c_1^2 - 2 \c_1 \c_6 + 9 \c_0 \c_7 + 3 \c_7 \c_5 - \c_6^2) \\
\xi_1  = & - \frac 1 6     (12 \c_3 \c_4 + 3 \c_2 \c_0 - 3 \c_2 \c_5 - \c_1^2 + 2 \c_1 \c_6 - 3 \c_0 \c_7 + 3 \c_7 \c_5 - \c_6^2) \\
\xi_2 = &  - \frac 1 {72}   (72 \c_3 \c_1 \c_4 + 27 \c_3 \c_0^2 - 54 \c_3 \c_0 \c_5 - 72 \c_3 \c_6 \c_4 + 27 \c_3 \c_5^2 - 27 \c_2^2 \c_4  - 9 \c_2 \c_1 \c_0 \\
& + 9 \c_2 \c_1 \c_5  + 9 \c_2 \c_0 \c_6 + 54 \c_2 \c_7 \c_4 - 9 \c_2 \c_6 \c_5 + 2 \c_1^3 - 6 \c_1^2 \c_6 + 9 \c_1 \c_0 \c_7  - 9 \c_1 \c_7 \c_5 \\
& + 6 \c_1 \c_6^2  - 9 \c_0 \c_7 \c_6  - 27 \c_7^2 \c_4 + 9 \c_7 \c_6 \c_5 - 2 \c_6^3) \\
\xi_3  = &  \frac 1 3  (27 \c_3 \c_0^2 + 18 \c_3 \c_0 \c_5 + 3 \c_3 \c_5^2 - 3 \c_2^2 \c_4 - 9 \c_2 \c_1 \c_0 + \c_2 \c_1 \c_5 - 15 \c_2 \c_0 \c_6 + 9 \c_7 \c_6 \c_5 \\
& - 2 \c_6^3  - 18 \c_2 \c_7 \c_4 - \c_2 \c_6 \c_5 + 2 \c_1^3 + 2 \c_1^2 \c_6 + 9 \c_1 \c_0 \c_7 + 15 \c_1 \c_7 \c_5 - 2 \c_1 \c_6^2 \\
& - 9 \c_0 \c_7 \c_6  - 27 \c_7^2 \c_4 )  \\
\xi_4  = & - \frac 1 9 (18 \c_3 \c_2 \c_0 \c_4 + 6 \c_3 \c_2 \c_5 \c_4 + 12 \c_3 \c_1^2 \c_4 - 36 \c_3 \c_1 \c_0 \c_5 + 24 \c_3 \c_1 \c_6 \c_4  - 12 \c_3 \c_1 \c_5^2 \\
& + 54 \c_3 \c_0^2 \c_6  + 54 \c_3 \c_0 \c_7 \c_4 + 18 \c_3 \c_7 \c_5 \c_4 + 12 \c_3 \c_6^2 \c_4 - 6 \c_3 \c_6 \c_5^2  - 6 \c_2^2 \c_1 \c_4 + 9 \c_2^2 \c_0^2 \\
& + 6 \c_2^2 \c_0 \c_5 - 12 \c_2^2 \c_6 \c_4 + 3 \c_2^2 \c_5^2  - 6 \c_2 \c_1^2 \c_0 + 2 \c_2 \c_1^2 \c_5  - 6 \c_2 \c_1 \c_0 \c_6 + 2 \c_2 \c_1 \c_6 \c_5 \\
& - 18 \c_2 \c_0^2 \c_7 - 24 \c_2 \c_0 \c_7 \c_5 + 6 \c_2 \c_0 \c_6^2 - 36 \c_2 \c_7 \c_6 \c_4  + 6 \c_2 \c_7 \c_5^2 + 2 \c_2 \c_6^2 \c_5 + \c_1^4 \\
& + 6 \c_1^2 \c_0 \c_7 + 6 \c_1^2 \c_7 \c_5 - 2 \c_1^2 \c_6^2 - 6 \c_1 \c_0 \c_7 \c_6  + 54 \c_1 \c_7^2 \c_4 - 6 \c_1 \c_7 \c_6 \c_5 \\
& + 27 \c_0^2 \c_7^2 - 18 \c_0 \c_7^2 \c_5 + 6 \c_0 \c_7 \c_6^2 + 9 \c_7^2 \c_5^2 - 6 \c_7 \c_6^2 \c_5 + \c_6^4) \\
\xi_5  = & - \frac 1 4 (36 \c_3^2 \c_1 \c_0^2 \c_4 + 24 \c_3^2 \c_1 \c_0 \c_5 \c_4 + 4 \c_3^2 \c_1 \c_5^2 \c_4 + 27 \c_3^2 \c_0^4 + 36 \c_3^2 \c_0^2 \c_6 \c_4 - 18 \c_3^2 \c_0^2 \c_5^2  \\
&  + 24 \c_3^2 \c_0 \c_6 \c_5 \c_4 - 8 \c_3^2 \c_0 \c_5^3 + 4 \c_3^2 \c_6 \c_5^2 \c_4 - \c_3^2 \c_5^4 + 4 \c_3 \c_2^2 \c_1 \c_4^2 - 6 \c_3 \c_2^2 \c_0^2 \c_4 - 8 \c_3 \c_2^2 \c_0 \c_5 \c_4\\
&  + 4 \c_3 \c_2^2 \c_6 \c_4^2 - 2 \c_3 \c_2^2 \c_5^2 \c_4 - 8 \c_3 \c_2 \c_1^2 \c_0 \c_4 - 8 \c_3 \c_2 \c_1^2 \c_5 \c_4 - 18 \c_3 \c_2 \c_1 \c_0^3 - 18 \c_3 \c_2 \c_1 \c_0^2 \c_5 \\
& - 16 \c_3 \c_2 \c_1 \c_0 \c_6 \c_4 + 2 \c_3 \c_2 \c_1 \c_0 \c_5^2 + 24 \c_3 \c_2 \c_1 \c_7 \c_4^2 - 16 \c_3 \c_2 \c_1 \c_6 \c_5 \c_4 + 2 \c_3 \c_2 \c_1 \c_5^3 + 18 \c_3 \c_2 \c_0^3 \c_6 \\
& + 6 \c_3 \c_2 \c_0^2 \c_6 \c_5 - 24 \c_3 \c_2 \c_0 \c_7 \c_5 \c_4 - 8 \c_3 \c_2 \c_0 \c_6^2 \c_4 + 6 \c_3 \c_2 \c_0 \c_6 \c_5^2 + 24 \c_3 \c_2 \c_7 \c_6 \c_4^2 - 8 \c_3 \c_2 \c_7 \c_5^2 \c_4\\
&  - 8 \c_3 \c_2 \c_6^2 \c_5 \c_4 + 2 \c_3 \c_2 \c_6 \c_5^3 + 4 \c_3 \c_1^3 \c_0^2 + 8 \c_3 \c_1^3 \c_0 \c_5 + 4 \c_3 \c_1^3 \c_5^2 - 12 \c_3 \c_1^2 \c_0^2 \c_6 + 24 \c_3 \c_1^2 \c_0 \c_7 \c_4\\
& - 8 \c_3 \c_1^2 \c_0 \c_6 \c_5 - 8 \c_3 \c_1^2 \c_7 \c_5 \c_4 + 4 \c_3 \c_1^2 \c_6 \c_5^2 + 18 \c_3 \c_1 \c_0^3 \c_7 - 6 \c_3 \c_1 \c_0^2 \c_7 \c_5 + 48 \c_3 \c_1 \c_0 \c_7 \c_6 \c_4 \\
& + 14 \c_3 \c_1 \c_0 \c_7 \c_5^2 - 16 \c_3 \c_1 \c_0 \c_6^2 \c_5 + 36 \c_3 \c_1 \c_7^2 \c_4^2 - 16 \c_3 \c_1 \c_7 \c_6 \c_5 \c_4 + 6 \c_3 \c_1 \c_7 \c_5^3 + 18 \c_3 \c_0^3 \c_7 \c_6 \\
& + 54 \c_3 \c_0^2 \c_7^2 \c_4 - 42 \c_3 \c_0^2 \c_7 \c_6 \c_5 + 16 \c_3 \c_0^2 \c_6^3 + 24 \c_3 \c_0 \c_7 \c_6^2 \c_4 - 10 \c_3 \c_0 \c_7 \c_6 \c_5^2 + 36 \c_3 \c_7^2 \c_6 \c_4^2 \\
& - 6 \c_3 \c_7^2 \c_5^2 \c_4 - 8 \c_3 \c_7 \c_6^2 \c_5 \c_4 + 2 \c_3 \c_7 \c_6 \c_5^3 - \c_2^4 \c_4^2 + 2 \c_2^3 \c_1 \c_0 \c_4 + 2 \c_2^3 \c_1 \c_5 \c_4 + 4 \c_2^3 \c_0^3 + 4 \c_2^3 \c_0^2 \c_5 \\
& + 6 \c_2^3 \c_0 \c_6 \c_4 - 8 \c_2^3 \c_7 \c_4^2 + 2 \c_2^3 \c_6 \c_5 \c_4 - \c_2^2 \c_1^2 \c_0^2 - 2 \c_2^2 \c_1^2 \c_0 \c_5 - \c_2^2 \c_1^2 \c_5^2 + 2 \c_2^2 \c_1 \c_0^2 \c_6  \\
& - 10 \c_2^2 \c_1 \c_0 \c_7 \c_4 + 6 \c_2^2 \c_1 \c_7 \c_5 \c_4 + 4 \c_2^2 \c_1 \c_6^2 \c_4 - 2 \c_2^2 \c_1 \c_6 \c_5^2 - 12 \c_2^2 \c_0^3 \c_7 - 4 \c_2^2 \c_0^2 \c_7 \c_5 + 3 \c_2^2 \c_0^2 \c_6^2 \\
& + 14 \c_2^2 \c_0 \c_7 \c_6 \c_4 - 8 \c_2^2 \c_0 \c_7 \c_5^2 + 2 \c_2^2 \c_0 \c_6^2 \c_5 - 18 \c_2^2 \c_7^2 \c_4^2 + 2 \c_2^2 \c_7 \c_6 \c_5 \c_4 + 4 \c_2^2 \c_6^3 \c_4 - \c_2^2 \c_6^2 \c_5^2 \\
& + 2 \c_2 \c_1^2 \c_0^2 \c_7 + 4 \c_2 \c_1^2 \c_0 \c_7 \c_5 - 16 \c_2 \c_1^2 \c_7 \c_6 \c_4 + 2 \c_2 \c_1^2 \c_7 \c_5^2 - 8 \c_2 \c_1 \c_0^2 \c_7 \c_6 - 42 \c_2 \c_1 \c_0 \c_7^2 \c_4 \\
& + 8 \c_2 \c_1 \c_0 \c_7 \c_6 \c_5 + 6 \c_2 \c_1 \c_7^2 \c_5 \c_4 - 8 \c_2 \c_1 \c_7 \c_6^2 \c_4 + 24 \c_2 \c_0^2 \c_7^2 \c_5 - 10 \c_2 \c_0^2 \c_7 \c_6^2 - 6 \c_2 \c_0 \c_7^2 \c_6 \c_4 \\
& - 4 \c_2 \c_0 \c_7^2 \c_5^2 + 4 \c_2 \c_0 \c_7 \c_6^2 \c_5 - 18 \c_2 \c_7^2 \c_6 \c_5 \c_4 + 4 \c_2 \c_7^2 \c_5^3 + 8 \c_2 \c_7 \c_6^3 \c_4 - 2 \c_2 \c_7 \c_6^2 \c_5^2 + 16 \c_1^3 \c_7^2 \c_4 \\
& + 3 \c_1^2 \c_0^2 \c_7^2 - 10 \c_1^2 \c_0 \c_7^2 \c_5 + 3 \c_1^2 \c_7^2 \c_5^2 + 6 \c_1 \c_0^2 \c_7^2 \c_6 + 18 \c_1 \c_0 \c_7^3 \c_4 - 8 \c_1 \c_0 \c_7^2 \c_6 \c_5 + 18 \c_1 \c_7^3 \c_5 \c_4 \\
& - 12 \c_1 \c_7^2 \c_6^2 \c_4 + 2 \c_1 \c_7^2 \c_6 \c_5^2 + 3 \c_0^2 \c_7^2 \c_6^2 + 18 \c_0 \c_7^3 \c_6 \c_4 - 12 \c_0 \c_7^3 \c_5^2 + 2 \c_0 \c_7^2 \c_6^2 \c_5 + 27 \c_7^4 \c_4^2   \\
& + 4 \c_7^3 \c_5^3 + 4 \c_7^2 \c_6^3 \c_4 - 18 \c_7^3 \c_6 \c_5 \c_4 - \c_7^2 \c_6^2 \c_5^2) 
\end{split}
\]
\[
\begin{split}
J_6 & =81 \c_3^4 \c_0^2 - 54 \c_3^3 \c_2 \c_1 \c_0 + 54 \c_3^3 \c_2 \c_0 \c_4 + 12 \c_3^3 \c_1^3 - 36 \c_3^3 \c_1^2 \c_4 + 54 \c_3^3 \c_1 \c_0 \c_5  - 108 \c_3^3 \c_1 \c_4^2 + 108 \c_3^3 \c_0^2 \c_6 \\
& + 378 \c_3^3 \c_0 \c_5 \c_4 + 324 \c_3^3 \c_4^3 + 12 \c_3^2 \c_2^3 \c_0 - 3 \c_3^2 \c_2^2 \c_1^2  + 6 \c_3^2 \c_2^2 \c_1 \c_4 - 36 \c_3^2 \c_2^2 \c_0 \c_5 + 45 \c_3^2 \c_2^2 \c_4^2 + 12 \c_7 \c_5^3 \c_4^2 \\
& + 6 \c_3^2 \c_2 \c_1^2 \c_5 - 126 \c_3^2 \c_2 \c_1 \c_0 \c_6  - 60 \c_3^2 \c_2 \c_1 \c_5 \c_4 - 162 \c_3^2 \c_2 \c_0 \c_6 \c_4 - 108 \c_3^2 \c_2 \c_0 \c_5^2 - 234 \c_3^2 \c_2 \c_5 \c_4^2 + 28 \c_3^2 \c_1^3 \c_6 \\
 & - 18 \c_3^2 \c_1^2 \c_0 \c_7 + 12 \c_3^2 \c_1^2 \c_6 \c_4 + 45 \c_3^2 \c_1^2 \c_5^2 - 108 \c_3^2 \c_1 \c_0 \c_7 \c_4 - 18 \c_3^2 \c_1 \c_0 \c_6 \c_5  - 252 \c_3^2 \c_1 \c_6 \c_4^2 + 150 \c_3^2 \c_1 \c_5^2 \c_4 \\
 &+ 54 \c_3^2 \c_0^2 \c_6^2 - 162 \c_3^2 \c_0 \c_7 \c_4^2 + 162 \c_3^2 \c_0 \c_6 \c_5 \c_4  - 60 \c_3^2 \c_0 \c_5^3 - 108 \c_3^2 \c_6 \c_4^3 + 45 \c_3^2 \c_5^2 \c_4^2 + 40 \c_3 \c_2^3 \c_0 \c_6 \\
 & - 10 \c_3 \c_2^2 \c_1^2 \c_6 + 48 \c_3 \c_2^2 \c_1 \c_0 \c_7  + 20 \c_3 \c_2^2 \c_1 \c_6 \c_4 + 144 \c_3 \c_2^2 \c_0 \c_7 \c_4 + 72 \c_3 \c_2^2 \c_0 \c_6 \c_5 + 150 \c_3 \c_2^2 \c_6 \c_4^2 - 10 \c_3 \c_2 \c_1^3 \c_7 \\
 & - 18 \c_3 \c_2 \c_1^2 \c_7 \c_4 - 28 \c_3 \c_2 \c_1^2 \c_6 \c_5 + 96 \c_3 \c_2 \c_1 \c_0 \c_7 \c_5 - 66 \c_3 \c_2 \c_1 \c_0 \c_6^2 + 162 \c_3 \c_2 \c_1 \c_7 \c_4^2 - 54 \c_7 \c_6 \c_5 \c_4^3   \\
 & - 104 \c_3 \c_2 \c_1 \c_6 \c_5 \c_4 + 288 \c_3 \c_2 \c_0 \c_7 \c_5 \c_4 - 126 \c_3 \c_2 \c_0 \c_6^2 \c_4 + 24 \c_3 \c_2 \c_0 \c_6 \c_5^2 + 378 \c_3 \c_2 \c_7 \c_4^3 + 81 \c_7^2 \c_4^4 \\
 & - 60 \c_3 \c_2 \c_6 \c_5 \c_4^2 - 22 \c_3 \c_1^3 \c_7 \c_5 + 20 \c_3 \c_1^3 \c_6^2 - 12 \c_3 \c_1^2 \c_0 \c_7 \c_6 - 126 \c_3 \c_1^2 \c_7 \c_5 \c_4 + 68 \c_3 \c_1^2 \c_6^2 \c_4 + 4 \c_0 \c_6^2 \c_5^3\\
 & + 6 \c_3 \c_1^2 \c_6 \c_5^2 - 72 \c_3 \c_1 \c_0 \c_7 \c_6 \c_4 + 48 \c_3 \c_1 \c_0 \c_7 \c_5^2 - 30 \c_3 \c_1 \c_0 \c_6^2 \c_5 - 162 \c_3 \c_1 \c_7 \c_5 \c_4^2 + 48 \c_0 \c_7 \c_6 \c_5^2 \c_4\\
 & + 12 \c_3 \c_1 \c_6^2 \c_4^2 + 20 \c_3 \c_1 \c_6 \c_5^2 \c_4 + 12 \c_3 \c_0^2 \c_6^3 - 108 \c_3 \c_0 \c_7 \c_6 \c_4^2 + 144 \c_3 \c_0 \c_7 \c_5^2 \c_4 - 18 \c_3 \c_0 \c_6^2 \c_5 \c_4  - 10 \c_0 \c_6^3 \c_5 \c_4 \\
 & - 8 \c_3 \c_0 \c_6 \c_5^3 + 54 \c_3 \c_7 \c_5 \c_4^3 - 36 \c_3 \c_6^2 \c_4^3 + 6 \c_3 \c_6 \c_5^2 \c_4^2 - 16 \c_2^4 \c_0 \c_7 + 4 \c_2^3 \c_1^2 \c_7 - 8 \c_2^3 \c_1 \c_7 \c_4  - 64 \c_2^3 \c_0 \c_7 \c_5 \\
 & + 12 \c_2^3 \c_0 \c_6^2 - 60 \c_2^3 \c_7 \c_4^2 + 20 \c_2^2 \c_1^2 \c_7 \c_5 - 3 \c_2^2 \c_1^2 \c_6^2 + 16 \c_2^2 \c_1 \c_0 \c_7 \c_6  + 12 \c_6^3 \c_4^3 - 3 \c_6^2 \c_5^2 \c_4^2  - 16 \c_0 \c_7 \c_5^4  \\
 & + 24 \c_2^2 \c_1 \c_7 \c_5 \c_4 + 6 \c_2^2 \c_1 \c_6^2 \c_4 + 48 \c_2^2 \c_0 \c_7 \c_6 \c_4 - 96 \c_2^2 \c_0 \c_7 \c_5^2 + 28 \c_2^2 \c_0 \c_6^2 \c_5 - 108 \c_2^2 \c_7 \c_5 \c_4^2  + 45 \c_2^2 \c_6^2 \c_4^2 \\
 & - 6 \c_2 \c_1^3 \c_7 \c_6 - 30 \c_2 \c_1^2 \c_7 \c_6 \c_4 + 28 \c_2 \c_1^2 \c_7 \c_5^2 - 10 \c_2 \c_1^2 \c_6^2 \c_5 + 32 \c_2 \c_1 \c_0 \c_7 \c_6 \c_5  - 10 \c_2 \c_1 \c_0 \c_6^3 - 18 \c_2 \c_1 \c_7 \c_6 \c_4^2 \\
 & + 72 \c_2 \c_1 \c_7 \c_5^2 \c_4 - 28 \c_2 \c_1 \c_6^2 \c_5 \c_4 + 96 \c_2 \c_0 \c_7 \c_6 \c_5 \c_4 - 64 \c_2 \c_0 \c_7 \c_5^3  - 22 \c_2 \c_0 \c_6^3 \c_4 + 20 \c_2 \c_0 \c_6^2 \c_5^2 + 54 \c_2 \c_7 \c_6 \c_4^3 \\
 & - 36 \c_2 \c_7 \c_5^2 \c_4^2 + 6 \c_2 \c_6^2 \c_5 \c_4^2 + \c_1^4 \c_7^2 + 12 \c_1^3 \c_7^2 \c_4  - 10 \c_1^3 \c_7 \c_6 \c_5 + 4 \c_1^3 \c_6^3 - 2 \c_1^2 \c_0 \c_7 \c_6^2 + 54 \c_1^2 \c_7^2 \c_4^2 \\
 & - 66 \c_1^2 \c_7 \c_6 \c_5 \c_4 + 12 \c_1^2 \c_7 \c_5^3 + 20 \c_1^2 \c_6^3 \c_4  - 3 \c_1^2 \c_6^2 \c_5^2 - 12 \c_1 \c_0 \c_7 \c_6^2 \c_4 + 16 \c_1 \c_0 \c_7 \c_6 \c_5^2 - 6 \c_1 \c_0 \c_6^3 \c_5 \\
 & + 108 \c_1 \c_7^2 \c_4^3 - 126 \c_1 \c_7 \c_6 \c_5 \c_4^2  + 40 \c_1 \c_7 \c_5^3 \c_4 + 28 \c_1 \c_6^3 \c_4^2 - 10 \c_1 \c_6^2 \c_5^2 \c_4 + \c_0^2 \c_6^4 - 18 \c_0 \c_7 \c_6^2 \c_4^2 
\end{split}
\]
\end{small}